\pdfoutput=1
\documentclass[12pt]{amsart}
\usepackage{lmodern}
\usepackage{ifthen}
\usepackage{amsfonts}
\usepackage{amsxtra}
\usepackage{amssymb}
\usepackage{mathdots}
\usepackage{array}
\usepackage[margin=1.0in]{geometry}
\usepackage{xcolor}
\definecolor{cite}{rgb}{0.30,0.60,1.00}
\definecolor{url}{rgb}{0.00,0.00,0.80}
\definecolor{link}{rgb}{0.40,0.10,0.20}
\usepackage[colorlinks,linkcolor=link,urlcolor=url,citecolor=cite,pagebackref,breaklinks]{hyperref}
\usepackage{bbm}
\usepackage{mathtools}
\usepackage{mathrsfs}
\usepackage{appendix}
\usepackage[all]{xy}
\usepackage[lite,abbrev,msc-links,alphabetic]{amsrefs}

\usepackage{graphicx}
\usepackage{multirow}
\usepackage{pstricks}
\usepackage{pst-pdf}
\usepackage{enumitem}
\usepackage{pifont}
\usepackage{marvosym}
\usepackage{txfonts}

\usepackage[OT2,T1]{fontenc}
\DeclareSymbolFont{cyrletters}{OT2}{wncyr}{m}{n}
\DeclareMathSymbol{\Sha}{\mathalpha}{cyrletters}{"58}

\usepackage{graphicx}

\makeatletter
\providecommand*{\Dashv}{%
  \mathrel{%
    \mathpalette\@Dashv\vDash
  }%
}
\newcommand*{\@Dashv}[2]{%
  \reflectbox{$\m@th#1#2$}%
}
\makeatother

\numberwithin{equation}{section}

\theoremstyle{plain}
\newtheorem{proposition}{Proposition}[section]
\newtheorem{conjecture}[proposition]{Conjecture}
\newtheorem{corollary}[proposition]{Corollary}
\newtheorem{lem}[proposition]{Lemma}
\newtheorem{theorem}[proposition]{Theorem}

\theoremstyle{definition}
\newtheorem{definition}[proposition]{Definition}

\newtheorem{notation}[proposition]{Notation}

\theoremstyle{remark}
\newtheorem{remark}[proposition]{Remark}
\newtheorem{example}[proposition]{Example}

\renewcommand{\b}[1]{\mathbf{#1}}
\renewcommand{\c}[1]{\mathcal{#1}}
\renewcommand{\d}[1]{\mathbb{#1}}
\newcommand{\f}[1]{\mathfrak{#1}}
\renewcommand{\r}[1]{\mathrm{#1}}
\newcommand{\s}[1]{\mathscr{#1}}
\renewcommand{\sf}[1]{\mathsf{#1}}
\renewcommand{\(}{\left(}
\renewcommand{\)}{\right)}
\newcommand{\res}{\mathbin{|}}
\newcommand{\ol}[1]{\overline{#1}{}}

\newcommand{\ul}{\underline}
\renewcommand{\leq}{\leqslant}
\renewcommand{\geq}{\geqslant}

\newcommand{\bD}{\b D}

\newcommand{\bG}{\b G}

\newcommand{\bM}{\b M}

\newcommand{\bP}{\b P}

\newcommand{\bm}{\b m}

\newcommand{\cD}{\c D}
\newcommand{\cE}{\c E}

\newcommand{\cP}{\c P}

\newcommand{\cS}{\c S}
\newcommand{\cT}{\c T}

\newcommand{\cV}{\c V}

\newcommand{\cX}{\c X}
\newcommand{\cY}{\c Y}

\newcommand{\dA}{\d A}
\newcommand{\dB}{\d B}
\newcommand{\dC}{\d C}
\newcommand{\dD}{\d D}
\newcommand{\dE}{\d E}
\newcommand{\dF}{\d F}

\newcommand{\dL}{\d L}

\newcommand{\dP}{\d P}
\newcommand{\dQ}{\d Q}

\newcommand{\dT}{\d T}

\newcommand{\dZ}{\d Z}

\newcommand{\fZ}{\f Z}

\newcommand{\fc}{\f c}

\newcommand{\fj}{\f j}

\newcommand{\fm}{\f m}
\newcommand{\fn}{\f n}

\newcommand{\fp}{\f p}

\newcommand{\rA}{\r A}
\newcommand{\rB}{\r B}

\newcommand{\rD}{\r D}
\newcommand{\rE}{\r E}
\newcommand{\rF}{\r F}
\newcommand{\rG}{\r G}
\newcommand{\rH}{\r H}
\newcommand{\rI}{\r I}
\newcommand{\rJ}{\r J}

\newcommand{\rM}{\r M}

\newcommand{\rP}{\r P}

\newcommand{\rR}{\r R}
\newcommand{\rS}{\r S}
\newcommand{\rT}{\r T}
\newcommand{\rU}{\r U}
\newcommand{\rV}{\r V}
\newcommand{\rW}{\r W}
\newcommand{\rX}{\r X}

\newcommand{\rb}{\r b}
\newcommand{\rc}{\r c}
\newcommand{\rd}{\r d}

\renewcommand{\rm}{\r m}

\newcommand{\rs}{\r s}

\newcommand{\sA}{\s A}

\newcommand{\sD}{\s D}
\newcommand{\sE}{\s E}

\newcommand{\sG}{\s G}

\newcommand{\sI}{\s I}

\newcommand{\sL}{\s L}

\newcommand{\sO}{\s O}
\newcommand{\sP}{\s P}
\newcommand{\sQ}{\s Q}

\newcommand{\sS}{\s S}

\newcommand{\sV}{\s V}

\newcommand{\sfF}{\sf F}
\newcommand{\sfG}{\sf G}
\newcommand{\sfH}{\sf H}
\newcommand{\sfI}{\sf I}

\newcommand{\sfK}{\sf K}

\newcommand{\sfO}{\sf O}
\newcommand{\sfP}{\sf P}
\newcommand{\sfQ}{\sf Q}
\newcommand{\sfR}{\sf R}

\newcommand{\sfT}{\sf T}
\newcommand{\sfU}{\sf U}
\newcommand{\sfV}{\sf V}

\newcommand{\sfs}{\sf s}

\newcommand{\sfv}{\sf v}

\newcommand{\sfx}{\sf x}

\newcommand{\tI}{\mathtt{I}}

\newcommand{\tN}{\mathtt{N}}

\newcommand{\tR}{\mathtt{R}}

\newcommand{\tT}{\mathtt{T}}

\newcommand{\tc}{\mathtt{c}}

\newcommand{\ts}{\mathtt{s}}

\newcommand{\pres}[2]{\prescript{#1}{}{#2}}

\newcommand{\cind}{\rc\text{-}\r{ind}}
\newcommand{\cl}{\r{cl}}

\newcommand{\cris}{\r{cris}}

\newcommand{\dr}{\r{dR}}
\newcommand{\et}{{\acute{\r{e}}\r{t}}}

\newcommand{\gr}{\r{gr}}

\newcommand{\id}{\r{id}}
\newcommand{\Igs}{\sI\!\!g\!s}

\newcommand{\Inc}{\r{Inc}}

\newcommand{\loc}{\r{loc}}

\newcommand{\mix}{\r{mix}}
\newcommand{\mnm}{\r{min}}

\newcommand{\ram}{\r{ram}}

\newcommand{\Sht}{\r{Sht}}

\newcommand{\sing}{\r{sing}}

\newcommand{\univ}{\r{univ}}
\newcommand{\unr}{\r{unr}}

\DeclareMathOperator{\CH}{CH}

\DeclareMathOperator{\coker}{coker}

\DeclareMathOperator{\End}{End}

\DeclareMathOperator{\Gal}{Gal}

\DeclareMathOperator{\GL}{GL}

\DeclareMathOperator{\Hom}{Hom}

\DeclareMathOperator{\IM}{im}

\DeclareMathOperator{\Ind}{Ind}

\DeclareMathOperator{\Ker}{ker}

\DeclareMathOperator{\Lie}{Lie}
\DeclareMathOperator{\Map}{Map}

\DeclareMathOperator{\ord}{ord}

\DeclareMathOperator{\Res}{Res}
\DeclareMathOperator{\Sel}{Sel}

\DeclareMathOperator{\Spec}{Spec}

\DeclareMathOperator{\Sym}{Sym}

\DeclareMathOperator{\val}{val}

\begin{document}

\title{Survey on bounding Selmer groups for Rankin--Selberg motives}

\author{Yifeng Liu}
\address{Institute for Advanced Study in Mathematics, Zhejiang University, Hangzhou 310058, China}
\email{liuyf0719@zju.edu.cn}

\author{Yichao Tian}
\address{Morningside Center of Mathematics, AMSS, Chinese Academy of Sciences, Beijing 100190, China}
\email{yichaot@math.ac.cn}

\author{Liang Xiao}
\address{New Cornerstone Lab, Beijing International Center for Mathematical Research, Peking University, Beijing 100871, China}
\email{lxiao@bicmr.pku.edu.cn}

\author{Wei Zhang}
\address{Department of Mathematics, Massachusetts Institute of Technology, Cambridge MA 02139, United States}
\email{weizhang@mit.edu}

\author{Xinwen Zhu}
\address{Department of Mathematics, Stanford University, Stanford CA 94305, United States}
\email{zhuxw@stanford.edu}

\date{\today}
\subjclass[2020]{11G05, 11G18, 11G40, 11R34}

\maketitle

\tableofcontents

\section{Introduction}
\label{ss:1}

Solving Diophantine equations has always been a central theme in the study of mathematics, since the ancient Greeks. In the context of modern algebraic geometry, the problem is integrated into the study of Chow groups of varieties defined over number fields. Recall that for a smooth projective variety $X$ over a field $F$ and an integer $r\geq 0$, the Chow group of codimension $r$ cycles of $X$, denoted by $\CH^r(X)$, is defined as the abelian group generated by codimension $r$ points of $X$ modulo rational equivalences. For every prime number $\ell$ different from the characteristic of $F$, we have the cycle class map
\[
\cl_\ell\colon\CH^r(X)\to\rH^{2r}(X,\dQ_\ell(r))
\]
to the $\ell$-adic \'{e}tale cohomology of $X$, through which we wish to turn the study of $\CH^r(X)$ to that of the cohomology. The pro-finite \'{e}tale covering $X_{\ol{F}}\to X$ induces a finite decreasing filtration $\rF^\bullet\rH^{2r}(X,\dQ_\ell(r))$ on $\rH^{2r}(X,\dQ_\ell(r))$ via the Hoschchild--Serre spectral sequence, satisfying that
\[
\frac{\rF^i\rH^{2r}(X,\dQ_\ell(r))}{\rF^{i+1}\rH^{2r}(X,\dQ_\ell(r))}=\rH^i(F,\rH^{2r-i}(X_{\ol{F}},\dQ_\ell(r))).
\]
By taking the preimage under $\cl_\ell$, it induces a decreasing filtration $\rF^\bullet\CH^r(X)$ on $\CH^r(X)$, which \emph{a priori} depends on $\ell$ and is not necessarily exhaustive. When $F$ is a number field, Beilinson and Bloch made the following remarkable predictions in the 1980s (see \cite{Bei87}), commonly known as the \emph{Beilinson--Bloch conjecture}:
\begin{enumerate}
  \item $\rF^\bullet\CH^r(X)$ does not depend on $\ell$ and $\rF^2\CH^r(X)=\CH^r(X)_{\r{tor}}$;\footnote{It is conjectured that $\rF^2\rH^{2r}(X,\dQ_\ell(r))=0$ when $F$ is a number field \cite{Jan89}.}

  \item the induced map $(\CH^r(X)/\rF^1\CH^r(X))\otimes_\dZ\dQ_\ell\to\rH^0(F,\rH^{2r}(X_{\ol{F}},\dQ_\ell(r)))$ is an isomorphism;

  \item the induced map $\rF^1\CH^r(X)\otimes_\dZ\dQ_\ell\to\rH^1(F,\rH^{2r-1}(X_{\ol{F}},\dQ_\ell(r)))$ is injective, whose image coincides with the Bloch--Kato Selmer group $\rH^1_f(F,\rH^{2r-1}(X_{\ol{F}},\dQ_\ell(r)))$ (see Remark \ref{re:selmer}).
\end{enumerate}
Note that (2) is nothing but the Tate conjecture (\cite{Tat65}*{Conjecture~1} over number fields). The three predictions together imply that $\CH^r(X)$ has finite rank, which reveals a fundamental difference between geometry over number fields and that over general fields. For example, it is a well-known theorem of Mumford that $\CH^2(X)$ could have uncountable rank when $F=\dC$ and $X$ is in a certain class of surfaces.

From now on, we assume that $F$ is a number field. The Beilinson--Bloch conjecture motivates us to study $\rH^0(F,\rH^{2r}(X_{\ol{F}},\dQ_\ell(r)))$ and $\rH^1_f(F,\rH^{2r-1}(X_{\ol{F}},\dQ_\ell(r)))$. For the former one, its dimension can be computed through the order of poles of certain $L$-functions. More precisely, it is expected that
\[
\dim_{\dQ_\ell}\rH^0(F,\rH^{2r}(X_{\ol{F}},\dQ_\ell(r)))=-\ord_{s=1}L(s,\rH^{2r}(X_{\ol{F}},\dQ_\ell(r))),
\]
which is equivalent to \cite{Tat65}*{Conjecture~2} under Tate's Conjecture 1. For the latter one, the Beilinson--Bloch--Kato conjecture relates its dimension to the order of zeros of certain $L$-functions. More precisely, it is expected that
\[
\dim_{\dQ_\ell}\rH^1_f(F,\rH^{2r-1}(X_{\ol{F}},\dQ_\ell(r)))=\ord_{s=0}L(s,\rH^{2r-1}(X_{\ol{F}},\dQ_\ell(r))).
\]
The above conjectural formula can be made toward more general motives, rather than the full cohomology of varieties.

In this note, we sketch the proof of our recent result on bounding Bloch--Kato Selmer groups of the so-called Rankin--Selberg motives, focusing on the case of the tensor product of symmetric power of elliptic curves with analytic rank zero.

We start from a brief recall of the Beilinson--Bloch--Kato conjecture for more general motives, which generalizes the Birch and Swinnerton-Dyer conjecture. Let $F$ be a number field contained in $\ol\dQ$ -- \emph{the} algebraic closure of $\dQ$ in $\dC$. Consider a motive $\rV$ over $F$ with coefficients in $\dQ$, for which we may understand as a compatible system $\{\rV_\ell\}_\ell$ of Galois representations of $F$ with coefficients in $\dQ_\ell$ ``of geometric origin''. Then the Beilinson--Bloch--Kato conjecture \cite{BK90} predicts that
\[
\ord_{s=0}L(s,\rV)=\dim\rH^1_f(F,\rV_\ell)-\dim\rH^0(F,\rV_\ell)
\]
for every prime number $\ell$. Here, $L(s,\rV)$ denotes the $L$-function of $\rV$, or rather $\rV_\ell$ for every $\ell$, which is conjectured to have a meromorphic continuation to $\dC$ and satisfy a functional equation relating $L(-s,\rV^\vee(1))$; and $\rH^1_f(F,\rV_\ell)$ denotes the Bloch--Kato Selmer group of $\rV_\ell$ (see the remark right below), as a finite-dimensional $\dQ_\ell$-vector space.

\begin{remark}\label{re:selmer}
We recall that $\rH^1_f(F,\rV_\ell)$ is the subspace of $\rH^1(F,\rV_\ell)$ consisting of classes $c$ such that for every finite place $v$ of $F$, its localization $\loc_\ell(c)$ belongs to a certain subspace $\rH^1_f(F_v,\rV_\ell)$ of $\rH^1(F_v,\rV_\ell)$. When $v\nmid\ell$, $\rH^1_f(F_v,\rV_\ell)=\rH^1_\unr(F_v,\rV_\ell)$ consists of the unramified classes. When $v\mid\ell$ and $\rV_\ell$ is crystalline at $v$, $\rH^1_f(F_v,\rV_\ell)$ parameterizes crystalline extensions of $\rV_\ell$ by $\dQ_\ell$ at $v$; in general, $\rH^1_f(F_v,\rV_\ell)\coloneqq\Ker\(\rH^1(F_v,\rV_\ell)\to\rH^1(F_v,\rV_\ell\otimes\dB_\cris)\)$.

When $\rV_\ell=\rV_\ell(A)$ is the rational $\ell$-adic Tate module of an abelian variety $A$ over $F$, $\rH^1_f(F_v,\rV_\ell)$ is canonically isomorphic to $\(\varprojlim_m\Sel_{\ell^m}(A)\)\otimes_{\dZ_\ell}\dQ_\ell$, where $\Sel_M(A)$ denotes the $M$-Selmer group of $A$ over $F$ for every positive integer $M$.
\end{remark}

Fix a positive integer $n$ and a pair of rational elliptic curves $A=(A_0,A_1)$. Let $n_0$ and $n_1$ be the unique even and odd members in $\{n,n+1\}$, respectively. Recall that for every prime number $\ell$ and $\alpha\in\{0,1\}$, we have the $\ell$-adic Tate module $\rT_\ell(A_\alpha)=\varprojlim_{m}A_\alpha(\ol\dQ)[\ell^m]$ of $A_\alpha$ as a $\dZ_\ell[\Gal(\ol\dQ/\dQ)]$-module. Put
\begin{align*}
\rT^A_{0,\ell}&\coloneqq\(\Sym_{\dZ_\ell}^{n_0-1}\rT_\ell(A_0)\)(1-r_0),\\
\rT^A_{1,\ell}&\coloneqq\(\Sym_{\dZ_\ell}^{n_1-1}\rT_\ell(A_1)\)(-r_1),\\
\rT^A_\ell&\coloneqq\rT^A_{0,\ell}\otimes_{\dZ_\ell}\rT^A_{1,\ell}.
\end{align*}
Then $\rT^A_\ell$ is a $\dZ_\ell[\Gal(\ol\dQ/\dQ)]$-module of $\dZ_\ell$-rank $n(n+1)$, equipped with a canonical polarization $\rT^A_\ell\simeq(\rT^A_\ell)^\vee(1)$. Put $\rV^A_\ell\coloneqq\rT^A_\ell\otimes_{\dZ_\ell}\dQ_\ell$. Then we obtain a canonically polarized motive $\rV^A\coloneqq\{\rV^A_\ell\}_\ell$ over $\dQ$, and its base change $\rV^A_F$ to $F$. It is clear that $\rH^0(F,\rV^A_\ell)=0$ for every prime number $\ell$. In particular, when $n=1$, the Beilinson--Bloch--Kato conjecture recovers (part of) the BSD conjecture.

The following theorem is a combination of \cite{LTXZZ}*{Theorem~1.1.1} and \cites{NT1,NT2}.

\begin{theorem}\label{th:sym}
Suppose that
\begin{enumerate}[label=(\alph*)]
  \item $F$ is a solvable CM field;

  \item $[F:\dQ]>2$;

  \item $A_0$ and $A_1$ are not geometrically isogenous to each other;

  \item neither $A_0$ nor $A_1$ has complex multiplication over $\ol\dQ$.
\end{enumerate}
If the $L$-value $L(0,\rV^A_F)$ does not vanish, then $\rH^1_f(F,\rV^A_\ell)$ vanishes for all but finitely many prime numbers $\ell$.
\end{theorem}

In comparison, we mention that Tate's both Conjecture 1 and 2 in \cite{Tat65} hold when $F$ is totally real or imaginary CM and $X=A_0^{m_0}\times A_1^{m_1}$ for arbitrary elliptic curves $A_0,A_1$ over $F$ and integers $m_0\geq 0,m_1\geq 1$. This is essentially an application of recent theorems on the Hodge conjecture and the Mumford--Tate conjecture for product of low dimensional abelian varieties, potential automorphy theorems and Shahidi's theorem on the non-vanishing of Rankin--Selberg $L$-functions at the near central point; see \cite{LZ} for more details.

\begin{remark}
We briefly recall the main history (which is by no means complete) of results of the kind in Theorem \ref{th:sym} (in particular, the case of central critical value).
\begin{enumerate}
  \item When $n=1$ and $F=\dQ$, the result in Theorem \ref{th:sym} was obtained by the groundbreaking work of Kolyvagin \cite{Kol90}, who invented the technique known as the Euler system; the proof uses the Gross--Zagier formula \cite{GZ86} to pass to the rank $1$ case.

  \item In \cite{BD05}, Bertolini and Darmon created another method to produce cohomology classes to bound Selmer groups via level raising, which was later streamlined by Howard \cite{How06} known as the bipartite Euler system; it can be used to give an alternative proof of Kolyvagin's result (under a mild restriction) without using the Gross--Zagier formula.

  \item In \cite{Liu1}, one of us proved the result in Theorem \ref{th:sym} when $n=2$ and $F=\dQ$ (under some mild restrictions).

  \item In \cite{DR17}, Darmon and Rotger proved a similar result in the spirit of Theorem \ref{th:sym} for the motive $\{\rV_\ell(A)\otimes\rho_1\otimes\rho_2\}_\ell$, where $A$ is an elliptic curve over $\dQ$ and $\rho_1,\rho_2$ are two two-dimensional Artin representations of $\dQ$ satisfying certain conditions.

  \item In \cite{Liu2}, one of us proved a similar result in the spirit of Theorem \ref{th:sym} for the motive $\{(\otimes\Ind^F_\dQ\rV_\ell(A))(-1)\}_\ell$, where $F$ is a totally real cubic extension of $\dQ$ and $A$ is an elliptic curve over $F$ satisfying certain mild conditions. Here, $\otimes\Ind$ denotes the tensor induction.
\end{enumerate}
We remark that the cohomology classes used to bound Selmer groups in (3), (5), and \cite{LTXZZ} are all in the spirit of Bertolini--Darmon.
\end{remark}

\begin{remark}
During the preparation of this note, we learnt that Newton and Thorne have extended their results to modular elliptic curves without CM over arbitrary totally real fields \cite{NT3}. With this update, Theorem \ref{th:sym} holds in a more general setting: namely, we start from a pair of modular elliptic curves $(A_0,A_1)$ over $F^+$ -- the maximal totally real subfield of $F$ -- and conclude the same result without condition (a).
\end{remark}

We will walk through the main steps of the proof of Theorem \ref{th:sym} (under certain simplification) from Section \ref{ss:2} to Section \ref{ss:8}. In the last section, we explain that various maps of Hecke modules in the previous sections can be interpreted within the framework of categorical local Langlands correspondence recently proposed by one of us \cite{Zhu20}.

In order to reduce the technical burden, we will gradually introduce restrictions on the situation we consider. The first one is that from now on, we assume that $F$ is a CM field containing an imaginary quadratic field $F^-$.

We collect some notation throughout the survey.
\begin{itemize}
  \item Write $n_\alpha=2r_\alpha+\alpha$ for $\alpha=0,1$ (so that $\{n_0,n_1\}=\{n,n+1\}$).

  \item Denote by $F^+$ the maximal totally real subfield contained in $F$, with $\tc\in\Gal(F/F^+)$ the complex conjugation.

  \item Denote by $H$ the Hilbert class field of $F$ and $H^+\subseteq H$ its maximal totally real subfield.

  \item Let $\tau\colon F\hookrightarrow\dC$ be the natural inclusion $F\subseteq\ol\dQ\subseteq\dC$; and put $\Phi\coloneqq\{\tau'\colon F\hookrightarrow\dC\res\tau'\res_{F^-}=\tau\res_{F^-}\}$, which gives a CM type of $F$.

  \item Let $\Sigma^F$ be the set of (rational) prime factors of the discriminant of $F$.

  \item Let $\Sigma^A$ be the union of $\Sigma^F$ and the set of (rational) prime factors of the conductors of $A_0$ and $A_1$.

  \item For every prime number $p$, we fix an isomorphism $\dC\simeq\ol\dQ_p$ and denote by $\fp,\fp^-,\fp^+$ the ($p$-adic) places of $F,F^-,F^+$ induced by the inclusion $F\subseteq\ol\dQ_p$, respectively.
\end{itemize}

We warn the readers that our notation and enumeration of conditions here could be different from the corresponding items in \cite{LTXZZ}, as we have modified the setup so that it suits better to the particular situation we consider in this note.

\subsection*{Acknowledgements}

We thank the anonymous referee for careful reading and helpful comments.

\section{Vanishing of Selmer groups}
\label{ss:2}

In this section, we present the Galois-theoretical argument for how one can show the vanishing of Selmer groups, following the original idea of Kolyvagin \cite{Kol90}.

We start by introducing a set of restrictions on the prime $\ell$. Define the following polynomials in $\dZ[u,v]$:
\begin{align*}
L_0(u,v)&=1,\quad L_1(u,v)=v^2-2u,\\
L_j(u,v)&=v^{2j}-\sum_{k=1}^j\binom{2j}{k}u^k L_{j-k}(u,v),\\
L(u,v)&=\prod_{j=1}^{r_1}\((L_{2j}(u,v)+u^{2j+1}+u^{2j-1})\prod_{i=0}^{r_0}(u^{2r_0-2j}L_{2j}(u,v)-u^{2r_0+2i}-u^{2r_0-2i})\).
\end{align*}

We introduce the first set of restrictions on the prime number $\ell$ for the coefficients.
\begin{description}
  \item[(L1)] The prime $\ell$ does not belong to $\Sigma^A$.

  \item[(L2)] There exists $(u,v)\in\dZ^2$ such that $\ell\nmid L(u,v)$.

  \item[(L3)] The image of $\Gal(\ol\dQ/\widetilde{F})$ in $\GL_{\dZ_\ell}(\rT_\ell(A_0))\times\GL_{\dZ_\ell}(\rT_\ell(A_1))$ consists of all pairs $(\gamma_0,\gamma_1)$ with $\det\gamma_0=\det\gamma_1$, where $\widetilde{F}$ denotes the normal closure of $F$ in $\ol\dQ$.
\end{description}

It is elementary that (L1) and (L2) exclude only finitely many $\ell$. By Serre's theorem \cite{Ser72}*{Th\'{e}or\`{e}me~6}, (L3) also excludes only finitely many $\ell$ under Theorem \ref{th:sym}(c,d). Condition (L3) says that the image is the largest possible one; while (L2) says that the largest possible image is indeed large enough to have level-raising primes modulo $\ell$ (defined below).

Now we take a prime number $\ell$ satisfying (L1--3).

\begin{definition}\label{de:level_raising}
For every positive integer $m$, a prime number $p$ is a \emph{level-raising prime modulo $\ell^m$} (with respect to the data $(F,A,n)$) if it satisfies
\begin{itemize}
  \item[(P1)] $p$ is odd, is inert in $F^-$ and splits completely in $H^+$.

  \item[(P2)] $p$ does not belong to $\Sigma^A$.

  \item[(P3)] $a_p(A_0)^2\equiv(p+1)^2\mod\ell^m$.

  \item[(P4)] $\ell\nmid p\prod_{i=1}^{n_0}(1-(-p)^i)$.

  \item[(P5)] $\ell\nmid L(p,a_p(A_1))$.
\end{itemize}
\end{definition}

\begin{remark}
In Definition \ref{de:level_raising}, (P4) is equivalent to that $\ell$ does not divide the cardinality of the finite unitary group of rank $n_0$ with respect to $\dF_{p^2}/\dF_p$; (P5) is to make sure that the tensor product Galois module $\rT^A_\ell\otimes\dF_\ell$ does not have ``extra level raising congruence''. See Lemma \ref{le:local_galois} below for their implications.
\end{remark}

\begin{lem}\label{le:local_galois}
Suppose that $p$ is level-raising prime modulo $\ell^m$. Then
\[
\rH^1_\unr(F_\fp,\rT^A_\ell\otimes\dZ/\ell^m),\quad
\rH^1_\sing(F_\fp,\rT^A_\ell\otimes\dZ/\ell^m)
\]
are both free $\dZ/\ell^m$-modules of rank $1$.
\end{lem}

Recall that $\rH^1_\sing(F_\fp,-)$ is defined as the quotient $\rH^1(F_\fp,-)/\rH^1_\unr(F_\fp,-)$.

\begin{proof}
This is a direct consequence of Definition \ref{de:level_raising}. We leave it to readers as an exercise.
\end{proof}

\begin{definition}\label{de:annihilator}
A \emph{system of Selmer annihilators} for $\rT^A_\ell$ consists for every integer $m>m_\ell$ (for some nonnegative integer $m_\ell$ depending on $\ell$) and all but finitely many level-raising prime $p$ modulo $\ell^m$, an element $c_{m,p}\in\rH^1(F,\rT^A_\ell\otimes\dZ/\ell^m)$ satisfying
\begin{enumerate}[label=(\alph*)]
  \item $\loc_v(c_{m,p})\in\rH^1(F_v,\rT^A_\ell\otimes\dZ/\ell^m)$ is crystalline (that is, the reduction of an element in $\rH^1(F_v,\rT^A_\ell)$ whose image in $\rH^1(F_v,\rV^A_\ell)$ parameterizes a crystalline extension) for every place $v$ of $F$ above $\ell$;

  \item the image of $\loc_\fp(c_{m,p})$ in $\rH^1_\sing(F_\fp,\rT^A_\ell\otimes\dZ/\ell^m)$ has ($\ell$-power) order at least $\ell^{m-m_\ell}$;

  \item $\loc_v(c_{m,p})\in\rH^1(F_v,\rT^A_\ell\otimes\dZ/\ell^m)$ is unramified for every nonarchimedean place $v\neq\fp$ of $F$ not above $\Sigma^A\cup\{\ell\}$.
\end{enumerate}
\end{definition}

\begin{proposition}\label{pr:kolyvagin}
Suppose that $F$ and $A$ satisfy (a--d) in Theorem \ref{th:sym} (and the simplifying assumption that $F$ contains an imaginary quadratic field). Then $\rH^1_f(F,\rV^A_\ell)=0$ for every prime number $\ell$ satisfying (L1--3) such that a system of Selmer annihilators for $\rT^A_\ell$ exists.
\end{proposition}

Once we have this proposition, the remaining task for the proof of Theorem \ref{th:sym} is to show that, as long as $L(0,\rV^A_F)\neq 0$, a system of Selmer annihilators for $\rT^A_\ell$ exists for all but finitely many $\ell$.

\begin{proof}
First, we realize that for every place $v$ of $F$ above $\Sigma^A$, $\rH^1(F_v,\rV^A_\ell)$ vanishes since $\rV^A_\ell$ satisfies the ``purity of monodromy filtration'', which implies that $\rH^1(F_v,\rT^A_\ell)$ is a finite $\dZ_\ell$-module. Thus, there exists a positive integer $m'$ such that $\prod_{v\mid\Sigma^A}\rH^1(F_v,\rT^A_\ell)$ is annihilated by $\ell^{m'}$.

We prove the proposition by contradiction. Suppose that $\rH^1_f(F,\rV^A_\ell)\neq 0$. Take $m\coloneqq m_\ell+m'+1$. By (L2,L3) and the Chebotarev density theorem, there are infinitely many level-raising primes modulo $\ell^m$. Thus, we can take such a prime $p$ for which the element $c_{m,p}$ (simply written as $c_m$ in the proof below) in Definition \ref{de:annihilator} exists.

By (L3), $\rT^A_\ell$ is residually absolutely irreducible. By \cite{Liu1}*{Lemma~5.8}, one can find an element $s\in\rH^1_f(F,\rT^A_\ell)$ whose image $s_m$ in $\rH^1_f(F,\rT^A_\ell\otimes\dZ/\ell^m)$ satisfies
\begin{itemize}
  \item[(a')] $\loc_v(s_m)$ is crystalline for every place $v$ of $F$ above $\ell$;

  \item[(b')] $\loc_\fp(s_m)$ is a generator of $\rH^1_\unr(F_\fp,\rT^A_\ell\otimes\dZ/\ell^m)$;

  \item[(c')] $\loc_v(s_m)$ is unramified for every nonarchimedean place $v\neq\fp$ of $F$ not above $\Sigma^A\cup\{\ell\}$.
\end{itemize}

For every prime $v$ of $F$, we have the local Tate pairing
\[
\langle\;,\;\rangle_v\colon\rH^1(F_v,\rT^A_\ell\otimes\dZ/\ell^m)\times\rH^1(F_v,\rT^A_\ell\otimes\dZ/\ell^m)
\to\rH^2(F_v,\dZ/\ell^m(1))=\dZ/\ell^m.
\]
Since both $s_m$ and $c_m$ belong to $\rH^1(F,\rT^A_\ell\otimes\dZ/\ell^m)$, we have $\sum_v\langle s_m,c_m\rangle_v=0$. On the other hand, we have $\ell^{m'}\langle s_m,c_m\rangle_v=0$ for $v$ above $\Sigma^A$; $\langle s_m,c_m\rangle_v=0$ for $v$ above $\ell$ by (a) and (a'); and $\langle s_m,c_m\rangle_v=0$ for $v\neq\fp$ of $F$ not above $\Sigma^A\cup\{\ell\}$ by (c) and (c'). As a consequence, $\ell^{m'}\langle s_m,c_m\rangle_\fp=0$. Since $\langle\;,\;\rangle_\fp$ induces a perfect pairing between $\rH^1_\unr(F_\fp,\rT^A_\ell\otimes\dZ/\ell^m)$ and $\rH^1_\sing(F_\fp,\rT^A_\ell\otimes\dZ/\ell^m)$ by the Tate duality, this contradicts with (b) and (b') by Lemma \ref{le:local_galois} and the choice of $m$. The proposition is proved.
\end{proof}

We briefly explain the idea of constructing the class $c_{m,p}$. We will construct
\begin{itemize}
  \item a smooth projective scheme $\cP$ over $F$ of dimension $2n-1$,

  \item a closed subscheme $\cD$ of $\cP$ of codimension $n$,

  \item a commutative ring $\dT$ acting on $\cP$ via \'{e}tale correspondences and an ideal $\fn$ of $\dT$ such that $\dT/\fn$ is a local $\dZ/\ell^m$-algebra, satisfying that $\rH^{2n}(\cP_{\ol\dQ},\dZ_\ell(n))/\fn=0$ and $\rH^{2n-1}(\cP_{\ol\dQ},\dZ_\ell(n))/\fn$ is isomorphic to $(\rT^A_\ell\otimes\dZ/\ell^m)^{\oplus\mu}$ for some $\mu>0$.
\end{itemize}
Then we define $c_{m,p}$ to be a suitable factor of the Abel--Jacobi class of $\cD$ in $\rH^1(F,\rH^{2n-1}(\cP_{\ol\dQ},\dZ_\ell(n))/\fn)$. In fact, the scheme $\cP$ admits $\dT$-equivariant strictly semistable reduction at $\fp$ and $\dT$-equivalent smooth reductions at other primes not above $\Sigma^A$. Condition (b) of Definition \ref{de:annihilator} will be proved by a calculation relating the image in the singular quotient with a certain period integral (sum, in fact) on definite unitary groups, which we introduce in the next section.

\section{Automorphic input}
\label{ss:3}

In this section, we state the precise consequence of the nonvanishing of the central $L$-value we need in order to prove Theorem \ref{th:sym}, which is a part of the Gan--Gross--Prasad conjecture \cite{GGP12} for unitary groups. For us, we only need to consider two classes of hermitian spaces, specified below.

\begin{definition}
We say that a hermitian space $V$ over $F$ (with respect to the quadratic extension $F/F^+$) of rank $n$ is \emph{standard definite} or \emph{standard indefinite} if $V$ has signature $(n,0)$ or $(n-1,1)$ at the place $\tau\res_{F^+}$ and signature $(n,0)$ at all other real places of $F^+$.

For a hermitian space $V$ over $F$ of rank $n$, we put $V^\sharp\coloneqq V\oplus F\cdot e$, where $e$ is a vector with length $1$. Denote by $\rU(V)$ and $\rU(V^\sharp)$ the unitary groups of $V$ and $V^\sharp$ over $F^+$, respectively. Let $\Sigma$ be a finite set of prime numbers containing $\Sigma^F$. A \emph{$\Sigma$-level pair} (for $(V,V^\sharp)$) is a pair $(K,K^\sharp)$ in which $K$ and $K^\sharp$ are neat decomposable open compact subgroups of $\rU(V)(\dA^\infty_{F^+})$ and $\rU(V^\sharp)(\dA^\infty_{F^+})$, respectively, satisfying
\begin{itemize}
  \item $K$ is contained in $K^\sharp\cap\rU(V)(\dA^\infty_{F^+})$ (here we regard $\rU(V)$ as the stabilizer of $e$ hence a subgroup of $\rU(V^\sharp)$);

  \item for every finite place $v^+$ of $F^+$ \emph{not} above $\Sigma$, there exists a self-dual lattice $\Lambda_{v^+}$ of $V_{v^+}$ such that $K_{v^+}$ and $K^\sharp_{v^+}$ are the stabilizers of $\Lambda_{v^+}$ and $\Lambda^\sharp_{v^+}\coloneqq\Lambda_{v^+}\oplus O_{F_{v^+}}\cdot e$, respectively.
\end{itemize}

For a hermitian space $V$ over $F$ of rank $n$ and $\alpha=0,1$, we write $V_\alpha$ for the unique member in $\{V,V^\sharp\}$ that has rank $n_\alpha$. Similarly, for $\Sigma$-level pair $(K,K^\sharp)$, we order them as $(K_0,K_1)$.
\end{definition}

\begin{definition}\label{de:relevant}
We say that a complex representation $\Pi=\Pi_0\times\Pi_1$ of $\GL_{n_0}(\dA_F)\times\GL_{n_1}(\dA_F)$ is \emph{relevant} if
\begin{enumerate}
  \item $\Pi$ is an irreducible cuspidal automorphic representation;

  \item $\Pi^\vee\simeq\Pi\circ\tc$;

  \item for every archimedean place $v$ of $F$ and $\alpha=0,1$, $\Pi_{\alpha,v}$ is isomorphic to the (irreducible) principal series representation induced by the characters $(\arg^{1-n_\alpha},\arg^{3-n_\alpha},\dots,\arg^{n_\alpha-3},\arg^{n_\alpha-1})$, where $\arg\colon\dC^\times\to\dC^\times$ is the \emph{argument character} defined by the formula $\arg(z)\coloneqq z/\sqrt{z\ol{z}}$.
\end{enumerate}
\end{definition}

Take a relevant representation $\Pi$. Let $\dQ(\Pi)\subseteq\dC$ be the coefficient field of $\Pi$, which is a number field \cite{LTXZZ}*{\S3.1}. Denote by $\dZ(\Pi)$ the ring of integers of $\dQ(\Pi)$.

For every prime $\lambda$ of $\dQ(\Pi)$, we have the associated semisimple Galois representation $\rV^\Pi_{\alpha,\lambda}$ of $\Gal(\ol\dQ/F)$ with coefficients in $\dQ(\Pi)_\lambda$, twisted such that $(\rV^\Pi_{\alpha,\lambda})^\vee\simeq(\rV^\Pi_{\alpha,\lambda})^\tc(\alpha-1)$.

Let $\Sigma\coloneqq\Sigma(\Pi)$ be the smallest set of prime numbers containing $\Sigma^F$ such that $\Pi_v$ is unramified for every finite place $v$ of $F$ not above $\Sigma$. Then we have the Satake homomorphism
\[
\phi^\Pi_\alpha\colon\dT^\Sigma_\alpha\to\dZ(\Pi),
\]
where $\dT^\Sigma_\alpha$ denotes the abstract unitary Hecke algebra away from $\Sigma$ of rank $n_\alpha$.

The following proposition is part of the Gan--Gross--Prasad conjecture, proved in \cite{BPLZZ}.

\begin{proposition}\label{pr:ggp}
Suppose that $L(\frac{1}{2},\Pi)\neq 0$. Then we can find
\begin{itemize}
  \item a standard definite hermitian space $V$ over $F$ of rank $n$,

  \item a $\Sigma$-level pair $(K,K^\sharp)$,

  \item an element $f_\alpha\in\dZ[\rU(V_\alpha)(F^+)\backslash\rU(V_\alpha)(\dA^\infty_{F^+})/K_\alpha][\phi^\Pi_\alpha]$ for $\alpha=0,1$,
\end{itemize}
such that
\[
\sum_{\rU(V)(F^+)\backslash\rU(V)(\dA^\infty_{F^+})/K}f_0(h)f_1(h)\neq 0.
\]
\end{proposition}

Indeed, the hermitian space $V$ is unique up to isometry. However, we do not need this fact in our proof.

Now let $A=(A_0,A_1)$ be a pair of rational elliptic curves satisfying Theorem \ref{th:sym}(c,d). We explain how to attach $A$ a relevant representation $\Pi^A$ when $F^+/\dQ$ is \emph{solvable}. Combining the modularity of rational elliptic curves, recent breakthrough of Newton--Thorne on the automorphy of symmetric power of modular forms \cites{NT1,NT2} and the cyclic automorphic base change \cite{AC89}, we have for $\alpha=0,1$ a unique up to isomorphism cuspidal automorphic representation $\Pi^A_\alpha$ of $\GL_{n_\alpha}(\dA_F)$ satisfying
\begin{itemize}
  \item For every finite place $v$ of $F$ such that $A_\alpha$ has good reduction at its underlying prime number $p_v$, $\Pi^A_{\alpha,v}$ is an unramified representation of Satake parameters $\{\alpha^{f_v i}\beta^{f_v(n_\alpha-1-i)}\res 0\leq i\leq n_\alpha-1\}$, where $\alpha,\beta$ are the two roots of the polynomial $X^2-\sqrt{p_v}^{-1}a_{p_v}(A_\alpha)X+1$ and $f_v$ is the residual extension degree of $F_v/\dQ_{p_v}$.

  \item For every (complex) archimedean place $v$ of $F$, $\Pi^A_{\alpha,v}$ is the principal series representation induced by the characters $\{\arg^{1-n_\alpha},\arg^{3-n_\alpha},\dots,\arg^{n_\alpha-3},\arg^{n_\alpha-1}\}$.
\end{itemize}
Then it follows that $\Pi^A\coloneqq\Pi^A_0\times\Pi^A_1$ is a relevant representation with $\dQ(\Pi^A)=\dQ$ and $\Sigma(\Pi^A)\subseteq\Sigma^A$. Moreover, we have $\rT^A_{\alpha,\ell}\otimes\dQ=\rV^{\Pi^A}_{\alpha,\ell}$ and $L(s,\rV^A_F)=L(s+\frac{1}{2},\Pi^A)$.

\section{Unitary moduli schemes}
\label{ss:4}

We start to construct the ``annihilator'' classes $c_m$, which are realized on certain moduli spaces of abelian varieties with extra structures. From now on, we further assume that $F^+\neq\dQ$ and that $n$ is even, that is, $n=n_0$. In particular, $r_0=r_1$, which we just write as $r$.

Take a standard definite hermitian space $V$ over $F$ of rank $n$, a $\Sigma$-level pair $(K,K^\sharp)$ for some $\Sigma$ containing $\Sigma^F$, and a prime number $p\not\in\Sigma\cup\{2\}$ that is inert in $F^-$ and splits completely in $H^+$. In particular, we may fix a triple $(A_\star,i_\star,\lambda_\star)$ where
\begin{itemize}
  \item $A_\star$ is an abelian scheme over $\dZ_{p^2}$ of dimension $[F^+:\dQ]$;

  \item $i_\star\colon O_F\to\End(A_\star)$ is a CM structure of CM type $\Phi$;

  \item $\lambda_\star\colon A_\star\to A_\star^\vee$ is a $p$-principal polarization under which $i_\star$ turns the complex conjugation into the Rosati involution.
\end{itemize}

\begin{definition}\label{de:moduli}
For $\alpha=0,1$, we define a moduli problem $\bM_\alpha$ on the category of locally Noetherian schemes over $\dZ_{p^2}$ such that $\bM_\alpha(S)$ is the set of equivalence classes of quadruples $(A_\alpha,i_\alpha,\lambda_\alpha,\eta^p_\alpha)$ where
\begin{itemize}
  \item $A_\alpha$ is an abelian scheme over $S$ of dimension $n_\alpha[F^+:\dQ]$;

  \item $i_\alpha\colon O_F\to\End(A_\alpha)$ is an action of $O_F$ such that for every $a\in O_F$, the characteristic polynomial for the action of $i_\alpha(a)$ on the Lie algebra of $A_\alpha$ is given by
      \[
      (T-a^\tc)(T-a)^{n_\alpha-1}\prod_{\tau'\in\Phi\setminus\{\tau\}}(T-\tau'(a))^{n_\alpha};
      \]

  \item $\lambda_\alpha\colon A_\alpha\to A_\alpha^\vee$ is a polarization under which $i_\alpha$ turns the complex conjugation into the Rosati involution and such that $\Ker\lambda_\alpha[p^\infty]$ is finite flat of rank $p^2$ contained in $A_\alpha[\fp]$;

  \item $\eta^p_\alpha$ is a $K^p_\alpha$-level structure, that is, for a chosen geometric point $s$ on every connected component of $S$, a $\pi_1(S,s)$-invariant $K^p_\alpha$-orbit of isometries
      \[
      \eta^p_\alpha\colon V_\alpha\otimes_\dQ\dA^{\infty,p}\xrightarrow\sim
      \Hom_{F\otimes_\dQ\dA^{\infty,p}}^{\lambda_\star,\lambda_\alpha}
      (\rH^\et_1(A_{\star s},\dA^{\infty,p}),\rH^\et_1(A_{\alpha s},\dA^{\infty,p}))
      \]
      of hermitian spaces (see \cite{LTXZZ}*{Construction~3.4.4} for the detailed meaning of the right-hand side).
\end{itemize}
Clearly, we have a morphism
\[
\bm\colon\bM_0\to\bM_1
\]
sending $(A_0,i_0,\lambda_0,\eta^p_0)$ to $(A_0\times A_\star,i_0\times i_\star,\lambda_0\times\lambda_\star,\eta^p_0\oplus 1)$.
\end{definition}

We denote by $\rM'_\alpha\coloneqq\bM_\alpha\otimes_{\dZ_{p^2}}\dQ_{p^2}$ and $\rM_\alpha\coloneqq\bM_\alpha\otimes_{\dZ_{p^2}}\dF_{p^2}$ the generic fiber and the special fiber of $\bM_\alpha$, respectively. Similarly, we have the morphisms $\rm'\colon\rM'_0\to\rM'_1$ and $\rm\colon\rM_0\to\rM_1$. For a point $(A_\alpha,\iota_\alpha,\lambda_\alpha,\eta^p_\alpha)\in\rM_\alpha(S)$, the Hodge sequence
\[
0 \to \omega_{A_\alpha^\vee/S}\to\rH_1^\dr(A_\alpha/S)\to\Lie_{A_\alpha/S}\to 0
\]
admits a direct sum decomposition over $\Hom(F,\dC)$ via the action $i_\alpha$ (and the fixed isomorphism $\dC\simeq\ol\dQ_p$):
\[
0 \to\bigoplus_{\tau'\colon F\to\dC} \omega_{A_\alpha^\vee/S,\tau'}\to
\bigoplus_{\tau'\colon F\to\dC}\rH_1^\dr(A_\alpha/S)_{\tau'}\to
\bigoplus_{\tau'\colon F\to\dC}\Lie_{A_\alpha/S,\tau'}\to 0.
\]
Denote by
\begin{itemize}
  \item $\rM_\alpha^\circ$ the closed locus of $\rM_\alpha$ on which $\lambda_{\alpha*}\omega_{A_\alpha^\vee/S,\tau}=0$,

  \item $\rM_\alpha^\bullet$ the closed locus of $\rM_\alpha$ on which $\Ker(\lambda_{\alpha*}\colon\rH_1^\dr(A_\alpha/S)_{\ol\tau}\to\rH_1^\dr(A^\vee_\alpha/S)_{\ol\tau})$ is contained in $\omega_{A_\alpha^\vee/S,\ol\tau}$,

  \item $\rM_\alpha^\dag\coloneqq\rM_\alpha^\circ\cap\rM_\alpha^\bullet$.
\end{itemize}
Clearly, $\rm$ restricts to morphisms
\[
\rm^\circ\colon\rM_0^\circ\to\rM_1^\circ,\quad
\rm^\bullet\colon\rM_0^\bullet\to\rM_1^\bullet,\quad
\rm^\dag\colon\rM_0^\dag\to\rM_1^\dag.
\]
For future use, we put
\[
\rS_\alpha^\circ\coloneqq\rU(V_\alpha)(F^+)\backslash\rU(V_\alpha)(\dA_{F^+}^\infty)/K_\alpha,
\]
regarded as a discrete scheme over $\dF_{p^2}$ according to the context; and we have a similar map
\[
\rs^\circ\colon\rS_0^\circ \to\rS_1^\circ.
\]

\begin{proposition}\label{pr:moduli}
We have
\begin{enumerate}
  \item For $\alpha=0,1$, $\bM_\alpha$ is a projective strictly semistable scheme over $\dZ_{p^2}$ of relative dimension $n_\alpha-1$.

  \item For $\alpha=0,1$, $\rM_\alpha^\circ$, $\rM_\alpha^\bullet$, and $\rM_\alpha^\dag$ are all projective smooth schemes over $\dF_{p^2}$ of dimensions $n_\alpha-1$, $n_\alpha-1$, and $n_\alpha-2$, respectively.

  \item We have a natural diagram
      \[
      \xymatrix{
      \rM_0^\circ \ar[r]^-{\pi_0^\circ}\ar[d]_-{\rm^\circ} & \rS_0^\circ \ar[d]^-{\rs^\circ} \\
      \rM_1^\circ \ar[r]^-{\pi_1^\circ} & \rS_1^\circ
      }
      \]
      of schemes over $\dF_{p^2}$, such that every geometric fiber of $\pi_\alpha^\circ$ is a projective space (of dimension $n_\alpha-1$) inside which $\rM_\alpha^\dag$ is a Fermat hypersurface of degree $p+1$.

  \item Let $V'$ be the (unique up to isomorphism) standard indefinite hermitian space over $F$ of rank $n$ such that $V'\otimes_F\dA_F^{\tau,\fp}\simeq V\otimes_F\dA_F^{\tau,\fp}$ (and we fix such an isomorphism); put $V'_0\coloneqq V'$ and $V'_1\coloneqq V^{\prime\sharp}$. Choose an almost self-dual lattice $\Lambda'_{\fp^+}$ of $V'_{\fp^+}$; put $\Lambda'_{0,\fp^+}\coloneqq\Lambda'_{\fp^+}$ and $\Lambda'_{1,\fp^+}\coloneqq\Lambda^{\prime\sharp}_{\fp^+}$. Then we have a canonical isomorphism $\cS_\alpha\otimes_FF_\fp\simeq\rM'_\alpha$ of schemes over $F_\fp=\dQ_{p^2}$ for $\alpha=0,1$, where $\cS_\alpha$ denotes the usual Shimura variety for the unitary group $\Res_{F^+/\dQ}\rU(V'_\alpha)$ at level $K_\alpha^{\fp^+}K'_{\alpha,\fp^+}$ where $K'_{\alpha,\fp^+}$ denotes the stabilizer of $\Lambda'_{\alpha,\fp^+}$, under which the base change of the special morphism $\sigma\colon\cS_0\to\cS_1$ to $\dQ_{p^2}$ coincides with $\rm'$.
\end{enumerate}
\end{proposition}

\begin{proof}
Parts (1,2) are in \cite{LTXZZ}*{Theorem~5.2.5}. Part (3) is in \cite{LTXZZ}*{Theorem~5.3.4}. Part (4) is \cite{LTXZZ}*{Remark~5.2.8}.
\end{proof}

\begin{remark}
The picture described in Proposition \ref{pr:moduli} can be regarded as a generalization of a similar (but much simpler) one studied by Ribet \cite{Rib89}, namely, the special fiber $\rM$ of a Shimura curve defined by a rational division quaternion algebra $B$ at a prime $p$ ramified in $B$ (with a maximal level structure at $p$). In his case, Ribet also has a union $\rM=\rM^\circ\cup\rM^\bullet$ in which \emph{both} strata $\rM^\circ$ and $\rM^\bullet$ are $\dP^1$-fibrations over a Shimura set, while the intersection $\rM^\dag\coloneqq\rM^\circ\cap\rM^\bullet$ is the locus of superspecial points. In our picture, one sees the difference between the $\circ$-stratum and the $\bullet$-stratum -- one remains a projective space fibration over a Shimura set while the other one does not have such a structure as long as the dimension is great than one. Moreover, in Ribet's case, the entire special fiber $\rM$ is the basic locus; however in our case (as we will see below), the $\circ$-stratum is always in the basic locus, but the basic locus in the open $\bullet$-stratum (that is, the complement of the $\circ$-stratum in the special fiber) has dimension just above half of the dimension of the special fiber.
\end{remark}

\begin{remark}\label{re:moduli}
We explain how the morphism $\pi_\alpha^\circ$ is constructed. First, one can show that the $\dF_{p^2}$-schemes $\rS_\alpha^\circ$ canonically parameterizes quadruples $(A_\alpha^\circ,i_\alpha^\circ,\lambda_\alpha^\circ,\eta^{p\circ}_\alpha)$ that are similar to $(A_\alpha,i_\alpha,\lambda_\alpha,\eta^p_\alpha)$ except that we replace the polynomial $(T-a^\tc)(T-a)^{n_\alpha-1}\prod_{\tau'\in\Phi\setminus\{\tau\}}(T-\tau'(a))^{n_\alpha}$ by the polynomial $\prod_{\tau'\in\Phi}(T-\tau'(a))^{n_\alpha}$ and replace the rank of $\Ker\lambda_\alpha[p^\infty]$ being $p^2$ by $\lambda_\alpha^\circ$ being $p$-principal. Second, one shows that for every object $(A_\alpha,i_\alpha,\lambda_\alpha,\eta^p_\alpha)$ of $\rM_\alpha^\circ(S)$ (for $S$ over $\dF_{p^2}$), there exists a unique up to equivalence object $(A_\alpha^\circ,i_\alpha^\circ,\lambda_\alpha^\circ,\eta^{p\circ}_\alpha)$ of $\rS_\alpha^\circ(S)$ with an
$O_F$-linear isogeny $\beta\colon A_\alpha\to A^\circ_\alpha$ compatible with level structures and satisfying $\Ker\beta[p^\infty]\subseteq A[\fp]$ and $\lambda=\beta^\vee\circ\lambda^\circ\circ\beta$. Then the morphism $\pi_\alpha^\circ$ sends $(A_\alpha,i_\alpha,\lambda_\alpha,\eta^p_\alpha)$ to $(A_\alpha^\circ,i_\alpha^\circ,\lambda_\alpha^\circ,\eta^{p\circ}_\alpha)$.
\end{remark}

By the above remark, it is clear that $\rM_\alpha^\circ$ is contained in the supersingular locus of $\rM_\alpha$, that is, the closed locus where the $p$-divisible group $A_\alpha[\fp^\infty]$ is supersingular. Recall that $K_{\fp^+}$ is the stabilizer of some self-dual lattice $\Lambda_{\fp^+}$. We now study the (remaining) supersingular locus on $\rM_\alpha^\bullet$. To do this, we need to fix
\begin{itemize}
  \item an element $\varpi\in O_{F^+}$ that is totally positive and satisfies $\val_{\fp^+}(\varpi)=1$, and $\val_{\fp'}(\varpi)=0$ for every prime $\fp'\neq\fp^+$ of $F^+$ above $p$; and

  \item an $O_{F_\fp}$-lattice $\Lambda_{\fp^+}^\bullet$ satisfying $\Lambda_{\fp^+}\subseteq\Lambda_{\fp^+}^\bullet\subseteq p^{-1}\Lambda_{\fp^+}$ and $(\Lambda_{\fp^+}^\bullet)^\vee=p\Lambda_{\fp^+}^\bullet$.
\end{itemize}
Let $K_{\alpha,\fp^+}^\bullet$ be the stabilizer of $\Lambda_{\alpha,\fp^+}^\bullet$, where $\Lambda_{0,\fp^+}^\bullet\coloneqq\Lambda_{\fp^+}^\bullet$ and $\Lambda_{1,\fp^+}^\bullet\coloneqq\Lambda_{\fp^+}^{\bullet\sharp}$. Let $K_\alpha^\bullet$ be the subgroup of $\rU(V_\alpha)(\dA_{F^+}^\infty)$ after replacing $K_{\alpha,\fp^+}$ by $K_{\alpha,\fp^+}^\bullet$ in $K_\alpha$. Put
\[
\rS_\alpha^\bullet\coloneqq\rU(V_\alpha)(F^+)\backslash\rU(V_\alpha)(\dA_{F^+}^\infty)/K_\alpha^\bullet,
\]
regarded as a discrete scheme over $\dF_{p^2}$ according to the context; and we have a natural map
\[
\rs^\bullet\colon\rS_0^\bullet \to\rS_1^\bullet.
\]

Similar to Remark \ref{re:moduli}, $\rS_\alpha^\bullet$ also parameterizes certain quadruples. Indeed, as in \cite{LTXZZ}*{Construction~5.4.6}, for every locally Noetherian scheme $S$ over $\dF_{p^2}$, $\rS_\alpha^\bullet(S)$ is the set of equivalence classes of quadruples $(A_\alpha^\bullet,i_\alpha^\bullet,\lambda_\alpha^\bullet,\eta^{p\bullet}_\alpha)$ where
\begin{itemize}
  \item $A_\alpha^\bullet$ is an abelian scheme over $S$ of dimension $n_\alpha[F^+:\dQ]$;

  \item $i_\alpha^\bullet\colon O_F\to\End(A_\alpha^\bullet)$ is an action of $O_F$ such that for every $a\in O_F$, the characteristic polynomial for the action of $i_\alpha^\bullet(a)$ on the Lie algebra of $A_\alpha^\bullet$ is given by $\prod_{\tau'\in\Phi}(T-\tau'(a))^{n_\alpha}$;

  \item $\lambda_\alpha^\bullet\colon A_\alpha^\bullet\to A_\alpha^{\bullet\vee}$ is a polarization under which $i_\alpha^\bullet$ turns the complex conjugation into the Rosati involution and such that $\Ker\lambda_\alpha^\bullet[p^\infty]$ is finite flat of rank $p^{2\alpha}$ contained in $A_\alpha^\bullet[\fp]$;

  \item $\eta^{p\bullet}_\alpha$ is a $K^p_\alpha$-level structure, that is, for a chosen geometric point $s$ on every connected component of $S$, a $\pi_1(S,s)$-invariant $K^p_\alpha$-orbit of isometries
      \[
      \eta^{p\bullet}_\alpha\colon V_\alpha\otimes_\dQ\dA^{\infty,p}\xrightarrow\sim
      \Hom_{F\otimes_\dQ\dA^{\infty,p}}^{\varpi\lambda_\star,\lambda_\alpha}
      (\rH^\et_1(A_{\star s},\dA^{\infty,p}),\rH^\et_1(A_{\alpha s}^\bullet,\dA^{\infty,p}))
      \]
      of hermitian spaces (see \cite{LTXZZ}*{Construction~3.4.4} for the detailed meaning of the right-hand side).
\end{itemize}

We define $\rB_\alpha^\bullet$ to be the scheme over $\dF_{p^2}$ such that for every locally Noetherian scheme $S$ over $\dF_{p^2}$, $\rB_\alpha^\bullet(S)$ is the set of equivalence classes of nonuples $(A_\alpha,i_\alpha,\lambda_\alpha,\eta^p_\alpha;
A_\alpha^\bullet,i_\alpha^\bullet,\lambda_\alpha^\bullet,\eta^{p\bullet}_\alpha;\gamma)$ where
\begin{itemize}
  \item $(A_\alpha,i_\alpha,\lambda_\alpha,\eta^p_\alpha)\in\rM_\alpha^\bullet(S)$;

  \item $(A_\alpha^\bullet,i_\alpha^\bullet,\lambda_\alpha^\bullet,\eta^{p\bullet}_\alpha)\in\rS_\alpha^\bullet(S)$;

  \item $\gamma\colon A_\alpha\to A_\alpha^\bullet$ is an $O_F$-linear isogeny satisfying a list of conditions (see \cite{LTXZZ}*{Definition~5.4.2} for more details).
\end{itemize}
By ``adding'' $(A_\star,i_\star,\lambda_\star,1;A_\star,i_\star,\varpi\lambda_\star,1;\id_{A_\star})$, we obtain a morphism $\rb^\bullet\colon\rB_0^\bullet\to\rB_1^\bullet$, rendering the diagram
\begin{align}\label{eq:bullet}
\xymatrix{
\rS_0^\bullet \ar[d]_-{\rs^\bullet} & \rB_0^\bullet \ar[l]_-{\pi_0^\bullet}\ar[r]^-{\iota_0^\bullet}\ar[d]^-{\rb^\bullet} & \rM_0^\bullet \ar[d]^-{\rm^\bullet} \\
\rS_1^\bullet  & \rB_1^\bullet \ar[l]_-{\pi_1^\bullet}\ar[r]^-{\iota_1^\bullet} & \rM_1^\bullet
}
\end{align}
commutative, in which $\pi_\alpha^\bullet$ and $\iota_\alpha^\bullet$ are natural forgetful morphisms.\footnote{The morphism $\rb^\bullet$ only exists when $n$ is even; when $n$ is odd, it has to be replaced by a correspondence. This is the only simplification by assuming $n$ even at the beginning of this section.}

\begin{proposition}\label{pr:basic}
In the diagram \eqref{eq:bullet},
\begin{enumerate}
  \item the fibers of $\pi_\alpha^\bullet$ are certain (smooth projective) Deligne--Lusztig varieties of dimension $r$;

  \item $\iota_\alpha^\bullet$ is locally a closed immersion;

  \item the morphism $\rb^\bullet$ is locally an isomorphism.
\end{enumerate}
\end{proposition}

\begin{proof}
Parts (1,2) are in \cite{LTXZZ}*{Theorem~5.4.4}. Part (3) is in \cite{LTXZZ}*{Proposition~5.10.12}.
\end{proof}

\begin{remark}
Roughly speaking, the Deligne--Lusztig variety in Proposition \ref{pr:basic}(1) parameterizes pairs of $\dF_{p^2}$-subspaces $(H_1,H_2)$ of a (nondegenerate) hermitian space over $\dF_{p^2}$ of dimension $n$ satisfying
\[
\dim_{\dF_{p^2}}H_i=r+1-i,\quad H_2\subseteq H_1\cap H_1^\perp\subseteq H_1+H_1^\perp\subseteq H_2^\perp.
\]
For example, when $n=2$, every irreducible component of $\rB_\alpha^\bullet$ is isomorphic to $\dP^1$ (over $\dF_{p^2}$); when $n=4$, every irreducible component of $\rB_\alpha^\bullet$ is, up to a purely inseparable morphism, the blow-up of the Fermat surface (over $\dF_{p^2}$) along all $\dF_{p^2}$-points.
\end{remark}

\begin{remark}
When $n=2$, we relate the supersingular locus of $\rM^\bullet_1$, which is the union of $\rM^\dag_1$ and the image of $\iota^\bullet_1$, to the Bruhat--Tits tree $\cT$ of $\rU(3)$. Indeed, $\rM^\dag_1$ is the disjoint union of Fermat curves parameterized by the set of hyperspecial vertices of $\cT$; the image of $\iota^\bullet_1$ is the disjoint union of projective lines parameterized by the set of non-hyperspecial vertices of $\cT$; and a Fermat curve intersects (transversally) with a projective line if and only if the corresponding vertices are connected by an edge of $\cT$.
\end{remark}

Now we study the fiber product of $\rM_\alpha^\circ$ and $\rB_\alpha^\bullet$ over $\rM_\alpha$. Put $K_\alpha^\dag\coloneqq K_\alpha\cap K_\alpha^\bullet$ and
\[
\rS_\alpha^\dag\coloneqq\rU(V_\alpha)(F^+)\backslash\rU(V_\alpha)(\dA_{F^+}^\infty)/K_\alpha^\dag,
\]
regarded as a discrete scheme over $\dF_{p^2}$ (which is a closed subscheme of $\rS_\alpha^\circ\times\rS_\alpha^\bullet$) according to the context. We also have a natural map
\[
\rs^\dag\colon\rS_0^\dag \to\rS_1^\dag.
\]

\begin{proposition}
The composite morphism
\[
\rM_\alpha^\circ\times_{\rM_\alpha}\rB_\alpha^\bullet\to\rM_\alpha^\circ\times\rB_\alpha^\bullet
\xrightarrow{\pi_\alpha^\circ\times\pi_\alpha^\bullet}\rS_\alpha^\circ\times\rS_\alpha^\bullet
\]
factors through the closed subscheme $\rS_\alpha^\dag$, such that every geometric fiber of the induced morphism
\[
\pi_\alpha^\dag\colon\rM_\alpha^\circ\times_{\rM_\alpha}\rB_\alpha^\bullet\to\rS_\alpha^\dag
\]
is a projective space of dimension $r-1$.
\end{proposition}

\begin{proof}
This is \cite{LTXZZ}*{Theorem~5.3.4}.
\end{proof}

\begin{notation}\label{no:intertwining}
For every commutative ring $L$, we denote by
\[
\tT_\alpha^{\bullet\circ}\colon L[\rS_\alpha^\circ]\to L[\rS_\alpha^\bullet],\quad
\tT_\alpha^{\circ\bullet}\colon L[\rS_\alpha^\bullet]\to L[\rS_\alpha^\circ]
\]
the two maps induced by the correspondence $\rS_\alpha^\dag\subseteq\rS_\alpha^\circ\times\rS_\alpha^\bullet$ between sets.
\end{notation}

In the next four sections, we take
\begin{itemize}
  \item a relevant representation $\Pi$ as in Section \ref{ss:3} with $\Sigma=\Sigma(\Pi)$;

  \item a prime $\lambda$ of $\dQ(\Pi)$ whose underlying prime number $\ell$ does not divide $p(p^2-1)$.
\end{itemize}
Write $\dZ_\lambda\coloneqq\dZ(\Pi)_\lambda$ for short.

For every integer $m\geq 1$, we put
\[
\fn_\alpha^m\coloneqq\dT_\alpha^{\Sigma\cup\{p\}}\cap
\Ker\(\dT^\Sigma_\alpha\xrightarrow{\phi^\Pi_\alpha}\dZ(\Pi)\to \dZ(\Pi)/\lambda^m\)
\]
and simply write $\fn_\alpha$ for $\fn_\alpha^1$. For every scheme over $\dF_{p^2}$ written in the form $\rX^?_?$, we write $\ol\rX^?_?\coloneqq\rX^?_?\otimes_{\dF_{p^2}}\ol\dF_p$ to save space.

Since $\bM_\alpha$ is projective and strictly semistable over $\dZ_{p^2}$, we have the weight spectral sequence $(\rE^{p,q}_{\alpha,s},\rd^{p,q}_{\alpha,s})$ (where $s$ denotes the page number) abutting to $\rH^{p+q}(\rM'_\alpha\otimes_{\dQ_{p^2}}\ol\dQ_p,\dZ_\lambda(r))$. More precisely, by Proposition \ref{pr:moduli}(2), for a fixed integer $q$, the complex $\rE^{?,q}_{\alpha,1}$ only has three possibly nonzero terms
\[
\rH^{q-2}(\ol\rM_\alpha^\dag,\dZ_\lambda(r-1))\xrightarrow{\rd_{\alpha,1}^{-1,q}}
\rH^q(\ol\rM_\alpha^\circ,\dZ_\lambda(r))\oplus\rH^q(\ol\rM_\alpha^\bullet,\dZ_\lambda(r))
\xrightarrow{\rd_{\alpha,1}^{0,q}}\rH^q(\ol\rM_\alpha^\dag,\dZ_\lambda(r))
\]
with $p=-1,0,1$. Here, if we denote by $\rd_\alpha^\circ\colon\rM_\alpha^\dag\to\rM_\alpha^\circ$ and $\rd_\alpha^\bullet\colon\rM_\alpha^\dag\to\rM_\alpha^\bullet$ the natural closed immersions, then $\rd_{\alpha,1}^{-1,q}=((\rd_\alpha^\circ)_!,-(\rd_\alpha^\bullet)_!)$ and $\rd_{\alpha,1}^{0,q}=(\rd_\alpha^\circ)^*-(\rd_\alpha^\bullet)^*$. At last, we denote by $\xi_\alpha\in\rH^2(\ol\rM_\alpha^\circ,\dZ_\lambda(1))$ the hyperplane section for the projective space fibration $\pi_\alpha^\circ$ in Proposition \ref{pr:moduli}(3).

In the next three sections, we study more precise arithmetic properties of the (localized) cohomology of $\bM_\alpha$ individually for $\alpha=0,1$.

\section{Arithmetic input for odd rank}

In this section, we look at the situation where $\alpha=1$. We define two maps
\begin{align*}
\Inc^1_\dag\colon\rE^{0,2r}_{1,1}&\to\rH^{2r}(\ol\rM_1^\bullet,\dZ_\lambda(r))\xrightarrow{(\rd^\bullet_1)^*}
\rH^{2r}(\ol\rM_1^\dag,\dZ_\lambda(r))\xrightarrow{(\rd^\circ_1)_!}\rH^{2r+2}(\ol\rM_1^\circ,\dZ_\lambda(r+1))\\
&\xrightarrow{\cup\xi_1^{n-r-1}}\rH^{2n}(\ol\rM_1^\circ,\dZ_\lambda(n))
\xrightarrow{(\pi_1^\circ)_!}\rH^0(\ol\rS_1^\circ,\dZ_\lambda)=\dZ_\lambda[\rS_1^\circ],\\
\Inc^1_\bullet\colon\rE^{0,2r}_{1,1}&\to\rH^{2r}(\ol\rM_1^\bullet,\dZ_\lambda(r))
\xrightarrow{(\iota_1^\bullet)^*}\rH^{2r}(\ol\rB_1^\bullet,\dZ_\lambda(r))
\xrightarrow{(\pi_1^\bullet)_!}\rH^0(\ol\rS_1^\bullet,\dZ_\lambda)=\dZ_\lambda[\rS_1^\bullet].
\end{align*}

By \cite{LTXZZ}*{Lemma~5.9.2(3)}, the map $\tT^{\bullet\circ}_1\circ\Inc^1_\dag+(p+1)^2\Inc^1_\bullet$ vanishes on the image of $\rd_{1,1}^{-1,2r}$ and hence induces a map
$\tT^{\bullet\circ}_1\circ\Inc^1_\dag+(p+1)^2\Inc^1_\bullet\colon\rE^{0,2r}_{1,2}\to\dZ_\lambda[\rS_1^\bullet]$. Put
\[
\nabla^1\coloneqq\tT^{\circ\bullet}_1\circ(\tT^{\bullet\circ}_1\circ\Inc^1_\dag+(p+1)^2\Inc^1_\bullet)\colon
\rE^{0,2r}_{1,2}\to\dZ_\lambda[\rS_1^\circ].
\]

\begin{proposition}\label{pr:tate}
Suppose that
\begin{enumerate}[label=(\alph*)]
  \item $\rH^i(\rM'_1\otimes_{\dQ_{p^2}}\ol\dQ_p,\dZ_\lambda)_{\fn_1}=0$ for $i\neq n$ and $\rH^n(\rM'_1\otimes_{\dQ_{p^2}}\ol\dQ_p,\dZ_\lambda)_{\fn_1}$ is a finite free $\dZ_\lambda$-module;

  \item the Galois representation $\rV^\Pi_{1,\lambda}$ is residually absolutely irreducible;

  \item the Satake parameter of $\Pi_{1,\fp}$ modulo $\lambda$ contains $1$ exactly once and does not contain $-p$.
\end{enumerate}
Then we have
\begin{enumerate}
  \item $(\rE^{p,q}_{1,2})_{\fn_1}=0$ unless $(p,q)=(0,2r)$ (in particular, $(\rE^{p,q}_{1,s})_{\fn_1}$ degenerates from the second page);

  \item the natural map $\((\rE^{0,2r}_{1,2})_{\fn_1}\)^{\Gal(\ol\dF_p/\dF_{p^2})}
      \to\((\rE^{0,2r}_{1,2})_{\fn_1}\)_{\Gal(\ol\dF_p/\dF_{p^2})}$ is an isomorphism;

  \item the localization of $\nabla^1$ at $\fn_1$ induces an isomorphism
      \[
      \nabla^1_{\fn_1}\colon\((\rE^{0,2r}_{1,2})_{\fn_1}\)_{\Gal(\ol\dF_p/\dF_{p^2})}
      \xrightarrow\sim\dZ_\lambda[\rS_1^\circ]_{\fn_1}.
      \]
\end{enumerate}
Putting together, we conclude that $\nabla^1_{\fn_1}$ induces an isomorphism
\[
\rH^0(\dQ_{p^2},\rH^{2r}(\rM'_1\otimes_{\dQ_{p^2}}\ol\dQ_p,\dZ_\ell(r))_{\fn_1})
\xrightarrow\sim\dZ_\lambda[\rS_1^\circ]_{\fn_1}.
\]
\end{proposition}

\begin{proof}
This is \cite{LTXZZ}*{Lemma~6.2.2~\&~Theorem~6.2.3}.
\end{proof}

\section{Arithmetic input for even rank}

In this section, we look at the situation where $\alpha=0$. We define two maps
\begin{align*}
\Inc^0_\circ\colon\rE^{0,2r}_{0,1}&\to\rH^{2r}(\ol\rM_0^\circ,\dZ_\lambda(r))
\xrightarrow{\cup\xi_0^{n-r-1}}\rH^{2n-2}(\ol\rM_0^\circ,\dZ_\lambda(n-1))
\xrightarrow{(\pi_0^\circ)_!}\rH^0(\ol\rS_0^\circ,\dZ_\lambda)=\dZ_\lambda[\rS_0^\circ],\\
\Inc^0_\bullet\colon\rE^{0,2r}_{0,1}&\to\rH^{2r}(\ol\rM_0^\bullet,\dZ_\lambda(r))
\xrightarrow{(\iota_0^\bullet)^*}\rH^{2r}(\ol\rB_0^\bullet,\dZ_\lambda(r))
\xrightarrow{(\pi_0^\bullet)_!}\rH^0(\ol\rS_0^\bullet,\dZ_\lambda)=\dZ_\lambda[\rS_0^\bullet].
\end{align*}
Put
\[
\nabla^0\coloneqq\tT^{\circ\bullet}_0\circ(\tT^{\bullet\circ}_0\circ\Inc^0_\circ+(p+1)\Inc^0_\bullet)
\colon\rE^{0,2r}_{0,1}\to\dZ_\lambda[\rS_0^\circ].
\]

By Proposition \ref{pr:moduli}(3) and the fact that the cohomology of Fermat hypersurface of even dimension concentrates only in even degrees \cite{LTXZZ}*{Lemma~5.6.2(1)}, the spectral sequence $\rE^{p,q}_{0,s}$ degenerates from the second page (even before localization). Note that the identification $\rE^{-1,2r}_{0,1}=\rE^{1,2r-2}_{0,1}(-1)$ induces a map $\mu\colon\rE^{-1,2r}_{0,2}\to\rE^{1,2r-2}_{0,2}(-1)$, which is nothing but the monodromy map. Moreover, by \cite{LTXZZ}*{Lemma~5.9.3(6)}, there is a natural short exact sequence
\begin{align}\label{eq:singular}
0 \to \frac{\rE^{1,2r-2}_{0,2}(-1)}{\mu\rE^{-1,2r}_{0,2}} \to
\rH^1_\sing(\dQ_{p^2},\rH^{2r-1}(\rM'_0\otimes_{\dQ_{p^2}}\ol\dQ_p,\dZ_\lambda(r)))
\to \rH^{2r-1}(\ol\rM_0^\bullet,\dZ_\lambda(r-1))^{\Gal(\ol\dF_p/\dF_{p^2})} \to 0
\end{align}
of $\dZ_\lambda$-modules, deduced from the spectral sequence.

\begin{proposition}[Arithmetic level raising, almost self-dual case]\label{pr:level_raising}
Let $m$ be a positive integer. Suppose that
\begin{enumerate}[label=(\alph*)]
  \item $\rH^i(\rM'_0\otimes_{\dQ_{p^2}}\ol\dQ_p,\dZ_\lambda)_{\fn_0}=0$ for $i\neq n-1$ and $\rH^{n-1}(\rM'_0\otimes_{\dQ_{p^2}}\ol\dQ_p,\dZ_\lambda)_{\fn_0}$ is a finite free $\dZ_\lambda$-module;

  \item the Galois representation $\rV^\Pi_{0,\lambda}$ is residually absolutely irreducible;

  \item the Satake parameter of $\Pi_{0,\fp}$ modulo $\lambda$ contains $p$ exactly once and does not contain $-1$;

  \item the Satake parameter of $\Pi_{0,\fp}$ modulo $\lambda^m$ contains $p$;

  \item the natural quotient map
      \[
      \frac{\dZ_\lambda[\rS_0^\circ]}{\fn_0^m}\to\frac{(\dZ_\lambda/\lambda^m)[\rS_0^\circ]}{\Ker\phi_0^\Pi}
      \]
      is an isomorphism.
\end{enumerate}
Then the composition $\nabla^0\circ\rd^{-1,2r}_{0,1}$ induces a map
\[
\frac{\rE^{1,2r-2}_{0,2}(-1)}{\mu\rE^{-1,2r}_{0,2}}\to\frac{\dZ_\lambda[\rS_0^\circ]}{\fn_0^m};
\]
and such map is surjective.
\end{proposition}

\begin{proof}
Indeed, by \cite{LTXZZ}*{Proposition~6.3.1(4)}, $\nabla^0\circ\rd^{-1,2r}_{0,1}$ induces a surjective map
\[
\(\frac{\rE^{1,2r-2}_{0,2}(-1)}{\mu\rE^{-1,2r}_{0,2}}\)_{\fn_0}
\to\frac{\dZ_\lambda[\rS_0^\circ]_{\fn_0}}{\tN\dZ_\lambda[\rS_0^\circ]_{\fn_0}}
\]
for a certain spherical Hecke operator $\tN$ at $\fp^+$ (denoted as $(p+1)\tR^\circ_{N,\fp}-\tI^\circ_{N,\fp}$ in \emph{loc. cit.}). However, by \cite{LTXZZ}*{Proposition~B.3.5(2)} and (e), $\tN$ belongs to $\fn_0^m$. Thus, the proposition follows.
\end{proof}

The above proposition (or rather its proof) has a corollary that can be regarded as a certain ``Ihara lemma'' for definite unitary groups.

Before stating our Ihara lemma, let us recall the more familiar Ihara lemma for Shimura sets defined by a definite rational quaternion algebra $B$ unramified at $p$. Let $K$ be a decomposable open compact subgroup of $\widehat{B}^\times$ that is maximal at $p$ (hence $K_p\simeq\GL_2(\dZ_p))$, and let $K^\dag\subseteq K$ be a subgroup satisfying $(K^\dag)^p=K^p$ and that $K^\dag_p$ is Iwahori. Suppose that $\ell\nmid p^2-1$, so that we have a decomposition $\Ind_{K^\dag}^{K}\dZ_\ell=\dZ_\ell\oplus\Omega_\ell$ in which $\Omega_\ell$ is the ``Steinberg component''. Let $\rS\coloneqq B^\times\backslash\widehat{B}^\times/K$ and $\rS^\dag\coloneqq B^\times\backslash\widehat{B}^\times/K^\dag$ be the corresponding Shimura sets, with the natural projection $\pi\colon\rS^\dag\to\rS$ and the involution $\iota\colon\rS^\dag\to\rS^\dag$ over $\rS$. Then the Ihara lemma in this setting (see, for example, \cite{Tay89}*{Lemma~4}) says that if we localize the following composite map
\[
\Map_{K}\(\rS^\dag,\Omega_\ell\)\to
\Map_{K}\(\rS^\dag,\Ind_{K^\dag}^{K}\dZ_\ell\)
=\dZ_\ell[\rS^\dag]\xrightarrow{\pi_!\circ\iota_!}\dZ_\ell[\rS]
\]
at a non-Eisenstein maximal ideal of the spherical Hecke algebra (away from primes at which $K^\dag$ is not hyperspecial), then it is surjective.

To state our Ihara lemma, we need more notation. Suppose that $\ell\nmid p\prod_{i=1}^{n_0}(1-(-p)^i)$. Then the $\dZ_\lambda[K_0^\circ]$-module $\Ind_{K_0^\dag}^{K_0^\circ}\dZ_\lambda$ has a unique direct summand, which we denote as $\Omega_\lambda$, that is a free $\dZ_\lambda$-module of rank larger than $1$ and such that $\Omega_\lambda^{K_0^\rI}$ is a free $\dZ_\lambda$-module of rank $1$; here $K_0^\rI$ is a subgroup of $K_0^\dag$ such that $(K_0^\rI)^{\fp^+}=(K_0^\dag)^{\fp^+}$ and that $(K_0^\rI)_{\fp^+}$ is an Iwahori subgroup.

\begin{corollary}[Ihara lemma, definite case]\label{co:ihara_def}
Suppose that $\ell\nmid p\prod_{i=1}^{n_0}(1-(-p)^i)$ and that
\begin{enumerate}[label=(\alph*)]
  \item $\rH^i(\rM'_0\otimes_{\dQ_{p^2}}\ol\dQ_p,\dZ_\lambda)_{\fn_0}=0$ for $i\neq n-1$ and $\rH^{n-1}(\rM'_0\otimes_{\dQ_{p^2}}\ol\dQ_p,\dZ_\lambda)_{\fn_0}$ is a finite free $\dZ_\lambda$-module;

  \item the Galois representation $\rV^\Pi_{0,\lambda}$ is residually absolutely irreducible;

  \item the Satake parameter of $\Pi_{0,\fp}$ modulo $\lambda$ contains $p$ at most once and does not contain $-1$;

  \item the natural quotient map
      \[
      \frac{\dZ_\lambda[\rS_0^\circ]}{\fn_0}\to\frac{(\dZ_\lambda/\lambda)[\rS_0^\circ]}{\Ker\phi_0^\Pi}
      \]
      is an isomorphism.
\end{enumerate}
Then the map
\[
\Map_{K_0^\circ}\(\rS_0^\dag,\Omega_\lambda\)_{\fn_0}\to\dZ_\lambda[\rS_0^\bullet]_{\fn_0}
\]
is surjective, where $\Map_{K_0^\circ}\(\rS_0^\dag,\Omega_\lambda\)\to\dZ_\lambda[\rS_0^\bullet]$ is the composition
\[
\Map_{K_0^\circ}\(\rS_0^\dag,\Omega_\lambda\)\to
\Map_{K_0^\circ}\(\rS_0^\dag,\Ind_{K_0^\dag}^{K_0^\circ}\dZ_\lambda\)
=\dZ_\lambda[\rS_0^\dag]\to\dZ_\lambda[\rS_0^\bullet]
\]
in which the last map is the pushforward along the map $\rS_0^\dag\to\rS_0^\bullet$.
\end{corollary}

\begin{proof}
Indeed, this follows from the dual statement of \cite{LTXZZ}*{Corollary~6.3.5} and Nakayama's lemma.
\end{proof}

Suppose that in the initial data of Section \ref{ss:4}, $V$ is standard \emph{indefinite} instead of definite. We have the Shimura varieties $\cS_0^\circ$, $\cS_0^\bullet$ and $\cS_0^\dag$ over $F$ instead of Shimura sets $\rS_0^\circ$, $\rS_0^\bullet$ and $\rS_0^\dag$, respectively. Then $\cS_0^\circ$ has a canonical (projective) smooth model over $O_{F_\fp}=\dZ_{p^2}$, whose special fiber (over $\dF_{p^2}$) we denote by $\widetilde\cS_0^\circ$. Let $V'$, $\Lambda'_{\fp^+}$ and $K'_{0,\fp^+}$ be as in Proposition \ref{pr:moduli}(4) but now with $V'$ standard definite. Put $\rS'_0\coloneqq\rU(V')(F^+)\backslash\rU(V')(\dA_{F^+}^\infty)/K_0^{\fp^+}K'_{0,\fp^+}$. Then the set of irreducible components of the basic locus of $\widetilde\cS_0^\circ$ is canonically parameterized by $\rS'_0$. Since the basic locus is of pure of dimension $r-1$, we have the absolute cycle class map
\[
\dZ_\lambda[\rS'_0]\to\rH^{2r}(\widetilde\cS_0^\circ,\dZ_\lambda(r))
\]
over $\dF_{p^2}$. The following conjecture generalizes a theorem of Ribet \cite{Rib90} on the level-raising for modular curves (or rather unitary Shimura curves in our setting) at good places.

\begin{conjecture}[Arithmetic level raising, self-dual case]\label{co:level_raising}
Suppose that
\begin{enumerate}[label=(\alph*)]
  \item $\rH^i(\cS_0^\circ\otimes_F\ol\dQ,\dZ_\lambda)_{\fn_0}=0$ for $i\neq n-1$ and $\rH^{n-1}(\cS_0^\circ\otimes_F\ol\dQ_p,\dZ_\lambda)_{\fn_0}$ is a finite free $\dZ_\lambda$-module;

  \item the Galois representation $\rV^\Pi_{0,\lambda}$ is residually absolutely irreducible;

  \item the Satake parameter of $\Pi_{0,\fp}$ modulo $\lambda$ contains $p$ at most once.
\end{enumerate}
Then the localized map
\[
\dZ_\lambda[\rS'_0]_{\fn_0}\to\rH^{2r}(\widetilde\cS_0^\circ,\dZ_\lambda(r))_{\fn_0}
\]
is surjective.
\end{conjecture}

Suppose that $\ell\nmid p\prod_{i=1}^{n_0}(1-(-p)^i)$. Then $\Omega_\lambda$ can be regarded as a $\dZ_\lambda$-local system on $\cS_0^\circ$. In particular, we have maps
\[
\rH^i(\cS_0^\circ\otimes_F\ol\dQ,\Omega_\lambda)
\to\rH^i(\cS_0^\circ\otimes_F\ol\dQ,\Ind_{K_0^\dag}^{K_0^\circ}\dZ_\lambda)
=\rH^i(\cS_0^\dag\otimes_F\ol\dQ,\dZ_\lambda)
\to\rH^i(\cS_0^\bullet\otimes_F\ol\dQ,\dZ_\lambda),
\]
in which the last map is the pushforward along the finite \'{e}tale morphism $\cS_0^\dag\to\cS_0^\bullet$. The following conjecture generalizes the Ihara lemma for modular curves (or rather unitary Shimura curves in our setting).

\begin{conjecture}[Ihara lemma, indefinite case]\label{co:ihara}
Suppose that $\ell\nmid p\prod_{i=1}^{n_0}(1-(-p)^i)$ and that
\begin{enumerate}[label=(\alph*)]
  \item $\rH^i(\cS_0^\circ\otimes_F\ol\dQ,\dZ_\lambda)_{\fn_0}=0$ for $i\neq n-1$ and $\rH^{n-1}(\cS_0^\circ\otimes_F\ol\dQ,\dZ_\lambda)_{\fn_0}$ is a finite free $\dZ_\lambda$-module;

  \item the Galois representation $\rV^\Pi_{0,\lambda}$ is residually absolutely irreducible;

  \item the Satake parameter of $\Pi_{0,\fp}$ modulo $\lambda$ contains $p$ at most once.
\end{enumerate}
Then the localized map
\[
\rH^{2r-1}(\cS_0^\circ\otimes_F\ol\dQ,\Omega_\lambda)_{\fn_0}\to
\rH^{2r-1}(\cS_0^\bullet\otimes_F\ol\dQ,\dZ_\lambda)_{\fn_0}
\]
is surjective.
\end{conjecture}

In \cite{LTX}, we prove the following deduction.

\begin{proposition}
Suppose that $\ell\nmid p\prod_{i=1}^{n_0}(1-(-p)^i)$. If the map
\[
\rH^{2r-1}(\cS_0^\circ\otimes_F\ol\dQ,\Omega_\lambda)_{\fn_0}\to
\rH^{2r-1}(\cS_0^\bullet\otimes_F\ol\dQ,\dZ_\lambda)_{\fn_0}
\]
in Conjecture \ref{co:ihara} is surjective, then the map
\[
\dZ_\lambda[\rS'_0]_{\fn_0}\to\rH^{2r}(\widetilde\cS_0^\circ,\dZ_\lambda(r))_{\fn_0}
\]
in Conjecture \ref{co:level_raising} is surjective.
\end{proposition}

\begin{remark}
In \cite{LTX}, we are able to show that under certain conditions on $\Pi$, Corollary \ref{co:ihara_def} implies Conjecture \ref{co:ihara} hence Conjecture \ref{co:level_raising}.
\end{remark}

\section{Interlude: potential map and singular quotient}

In this section, we explain the motivation of the map $\nabla^0$ constructed in the previous section. The underlying idea is the so-called \emph{potential map} used to compute $\rH^1_\sing$ under a certain favorable situation for general varieties over local fields, which was developed by one of us in \cite{Liu2}. Again, let $n=2r$ be a positive even integer.

Let $K$ be a henselian discrete valuation field with residue field $\kappa$ and a separable closure $\overline{K}$, such that the characteristic of $\kappa$ is different from $\ell$. Let $\Lambda$ be either a $\dZ_\ell$-ring that is finitely generated as a $\dZ_\ell$-module or the quotient field of such a ring, which will serve as the ring of coefficients. We consider a proper strictly semistable scheme $X$ over $O_K$ of pure relative dimension $n-1$. In particular, the special fiber $X_\kappa\coloneqq X\otimes_{O_K}\kappa$ is a strict normal crossing divisor of $X$. Suppose that $\{X_1,\dots,X_m\}$ is the set of irreducible components of $X_\kappa$. For $p\geq 0$, put
\[
X^{(p)}_\kappa\coloneqq\coprod_{I\subset\{1,\dots,m\},|I|=p+1}\bigcap_{i\in I}X_i.
\]
Then $X^{(p)}_\kappa$ is a finite disjoint union of smooth proper $\kappa$-schemes of codimension $p$ in $X_\kappa$. We have the pullback map
\[
\delta_p^*\colon\rH^q(X^{(p)}_{\overline{\kappa}},\Lambda(j))
\to\rH^q(X^{(p+1)}_{\overline{\kappa}},\Lambda(j))
\]
and the pushforward (Gysin) map
\[
\delta_{p*}\colon\rH^q(X^{(p)}_{\overline{\kappa}},\Lambda(j))
\to\rH^{q+2}(X^{(p-1)}_{\overline{\kappa}},\Lambda(j+1))
\]
for every integer $j$ (see \cite{Liu2}*{\S2.1} for more details). These maps satisfy the formula
\begin{align*}
\delta_{p-1}^*\circ\delta_{p*}+\delta_{p+1*}\circ\delta_p^*=0.
\end{align*}

We have the weight spectral sequence $\rE^{p,q}_s$ attached to $X$ with coefficients $\Lambda(r)$ abutting to $\rH^{p+q}(X_{\overline{K}},\Lambda(r))$, whose first page is given by
\[
\rE^{p,q}_1=\bigoplus_{i\geq\max(0,-p)}\rH^{q-2i}(X_{\overline{\kappa}}^{(p+2i)},\Lambda(r-i)).
\]
We also have the monodromy map
\[
\mu^{p,q}_s\colon\rE^{p-1,q+1}_s\to\rE^{p+1,q-1}_s(-1)
\]
induced by the natural one on the first page.

Put
\[
\rB^r(X)\coloneqq
\Ker\(\delta_0^*\colon\rH^{2r}(X_{\overline{\kappa}}^{(0)},\Lambda(r))
\to\rH^{2r}(X_{\overline{\kappa}}^{(1)},\Lambda(r))\),
\]
and dually
\[
\rB_r(X)\coloneqq
\coker\(\delta_{1*}\colon\rH^{2(r-1)}(X_{\overline{\kappa}}^{(1)},\Lambda(r-1))
\to\rH^{2r}(X_{\overline{\kappa}}^{(0)},\Lambda(r))\).
\]
As a submodule of $\rE^{0,2r}_1$, $\rB^r(X)$ is contained in $\Ker\rd^{0,2r}_1$. Thus, it makes sense to put
\[
\rB^r(X)^0\coloneqq\Ker\(\rB^r(X)\to\rE^{0,2r}_2\),
\]
and dually the quotient $\rB_r(X)_0$ of $\rB_r(X)$. It is not hard to see that the identity map of $\rH^{2r}(X_{\overline{\kappa}}^{(0)},\Lambda(r))$ induces a map
\[
\rB_r(X)_0\to\rB^r(X)^0.
\]
Finally, denote by $\rA_r(X)_0$ and $\rA^r(X)^0$ the maximal $\Gal(\overline\kappa/\kappa)$-invariant submodule and quotient module of $\rB_r(X)_0$ and $\rB^r(X)^0$, respectively. We define the \emph{potential map} to be the composite map
\[
\Delta\colon\rA_r(X)_0\to\rB_r(X)_0\to\rB^r(X)^0\to\rA^r(X)^0.
\]

\begin{proposition}\label{pr:potential}
Suppose that
\begin{enumerate}[label=(\alph*)]
  \item $\rE^{p,q}_s$ degenerates from the second page;

  \item if $\rE^{p,n-1-p}_2(-1)$ has a non-trivial subquotient on which $\Gal(\overline\kappa/\kappa)$ acts trivially, then $p=1$;

  \item for every $\Gal(\overline\kappa/\kappa)$-stable subquotient $\Lambda$-module $\rM$ of $\ker\mu^{1,2r-1}_1\oplus\coker\mu^{-1,2r-1}_1$, the natural map
      $\rM^{\Gal(\overline\kappa/\kappa)}\to\rM_{\Gal(\overline\kappa/\kappa)}$ is an isomorphism.
\end{enumerate}
Then we have a canonical exact sequence
\begin{align*}
0 \to \rH^1_\unr(K,\rH^{2r-1}(X_{\overline{K}},\Lambda(r)))\to
\rA_r(X)_0 \xrightarrow{\Delta^r} \rA^r(X)^0 \xrightarrow{\eta^r}
\rH^1_\sing(K,\rH^{2r-1}(X_{\overline{K}},\Lambda(r))) \to 0
\end{align*}
of $\Lambda$-modules.
\end{proposition}

\begin{proof}
This is \cite{Liu2}*{Theorem~2.9}.
\end{proof}

When $X^{(2)}=\emptyset$, which is the case from the previous section where $K=\dQ_{p^2}$ and $X=\bM_0$, the map $\eta^r\colon\rA^r(X)^0 \to \rH^1_\sing(K,\rH^{2r-1}(X_{\overline{K}},\Lambda(r)))$ can be explained easily. In this case, the first page $\rE^{p,q}_1$ is as follows:
     \[
     \boxed{
     \xymatrix{
     q\geq 2r+1 & \cdots \ar[r] & \cdots \ar[r] & \cdots \\
     q=2r & \rH^{2r-2}(X^{(1)}_{\ol\kappa},\Lambda(r-1)) \ar[r]^-{\rd^{-1,2r}_1} & \rH^{2r}(X^{(0)}_{\ol\kappa},\Lambda(r)) \ar[r]^-{\rd^{0,2r}_1} &
     \rH^{2r}(X^{(1)}_{\ol\kappa},\Lambda(r)) \\
     q=2r-1 & \rH^{2r-3}(X^{(1)}_{\ol\kappa},\Lambda(r-1)) \ar[r]^-{\rd^{-1,2r-1}_1} & \rH^{2r-1}(X^{(0)}_{\ol\kappa},\Lambda(r)) \ar[r]^-{\rd^{0,2r-1}_1} &
     \rH^{2r-1}(X^{(1)}_{\ol\kappa},\Lambda(r)) \\
     q=2r-2 & \rH^{2r-4}(X^{(1)}_{\ol\kappa},\Lambda(r-1)) \ar[r]^-{\rd^{-1,2r-2}_1} & \rH^{2r-2}(X^{(0)}_{\ol\kappa},\Lambda(r)) \ar[r]^-{\rd^{0,2r-2}_1} &
     \rH^{2r-2}(X^{(1)}_{\ol\kappa},\Lambda(r)) \\
     q\leq 2r-3 & \cdots \ar[r] & \cdots \ar[r] & \cdots \\
     \rE^{p,q}_1 & p=-1 & p=0 & p=1
     }
     }
     \]
which is concentrated in the vertical strip $p\in\{-1,0,1\}$. It is not hard to see that the map $\rd^{-1,2r}_1$ induces an isomorphism
\[
\frac{\rE^{1,2r-2}_2(-1)}{\mu^{0,2r-1}_2\rE^{-1,2r}_2}\simeq
\frac{\IM\rd^{-1,2r}_1}{\IM\(\rd^{-1,2r}_1\circ\rd^{0,2r-2}(-1)\)}.
\]
On the other hand, we have a natural injection
\[
\frac{\rE^{1,2r-2}_2(-1)}{\mu^{0,2r-1}_2\rE^{-1,2r}_2}\hookrightarrow
\rH^1_\sing(K,\rH^{2r-1}(X_{\overline{K}},\Lambda(r)))
\]
which has already appeared in \eqref{eq:singular}. Since $\rA^r(X)^0$ is contained in $\IM\rd^{-1,2r}_1$ by construction, we obtain a natural map $\eta^r$.

Though in Proposition \ref{pr:level_raising} we only have (a,c) in Proposition \ref{pr:potential} and (b) is unclear yet, it still motivates us to define the map $\nabla^0$ as the cokernel of the potential map $\Delta$ in mind. The complication involved in (b) will be addressed in the next section, in which we will obtain a precise characterization of $\rH^1_\sing$.

\begin{example}
We give an example of Proposition \ref{pr:potential} for $X$ being a Mumford curve, that is, every irreducible component of $X_\kappa$ is a projective line over $\kappa$ (so that $r=1$). In this case, the first page $\rE^{p,q}_1$ is even simpler, which is just
     \[
     \boxed{
     \xymatrix{
     q=2 & \rH^0(X^{(1)}_{\ol\kappa},\Lambda) \ar[r]^-{\rd^{-1,2}_1} & \rH^2(X^{(0)}_{\ol\kappa},\Lambda(1)) \ar[r] & 0 \\
     q=1 & 0 \ar[r] & 0 \ar[r] & 0 \\
     q=0 & 0 \ar[r] & \rH^0(X^{(0)}_{\ol\kappa},\Lambda(1)) \ar[r]^-{\rd^{0,0}_1} &
     \rH^0(X^{(1)}_{\ol\kappa},\Lambda(1)) \\
     \rE^{p,q}_1 & p=-1 & p=0 & p=1
     }
     }
     \]

Denote by $\Gamma$ the dual graph of $X_\kappa$, which is a finite oriented graph where the orientation comes from a fixed ordering of the set of irreducible components of $X_\kappa$. Denote by $\cV_\Gamma$ and $\cE_\Gamma$ the sets of vertices and edges of $\Gamma$, identified with the sets of irreducible components of $X^{(0)}_\kappa$ and $X^{(1)}_\kappa$, respectively. We have two maps $v_0,v_1\colon\cE_\Gamma\to\cV_\Gamma$ assigning the starting and the ending vertices, respectively.
\begin{itemize}
  \item Define the ``restriction'' map $\rho\colon\Lambda[\cV_\Gamma]\to\Lambda[\cE_\Gamma]$ to be the one sending a function $F$ to $\rho F$ such that $(\rho F)(e)=F(v_1(e))-F(v_0(e))$.

  \item Define the ``Gysin'' map $\gamma\colon\Lambda[\cE_\Gamma]\to\Lambda[\cV_\Gamma]$ to be the one sending a function $G$ to $\gamma G$ such that $(\gamma G)(v)=\sum_{v_1(e)=v}G(e)-\sum_{v_0(e)=v}G(e)$.
\end{itemize}
Then the maps $\rd^{0,0}_1$ and $\rd^{-1,2}_1$ are canonically identified with $\rho$ and $\gamma$, respectively. Under such identification, $\rA^1(X)^0$ coincides with $\Lambda[\cV_\Gamma]^0$, the submodule of $\Lambda[\cV_\Gamma]$ consisting of functions with total sum zero; $\rA_1(X)_0$ coincides with $\Lambda[\cV_\Gamma]_0$, the quotient of $\Lambda[\cV_\Gamma]$ by constant functions. Then the potential map is simply
\[
\Delta=\gamma\circ\rho\colon\Lambda[\cV_\Gamma]_0\to\Lambda[\cV_\Gamma]^0.
\]
It is clear that the three conditions in Proposition \ref{pr:potential} are all satisfied; hence we have canonical isomorphisms
\begin{align*}
\rH^1_\unr(K,\rH^1(X_{\overline{K}},\Lambda(1)))&=
\ker\(\gamma\circ\rho\colon\Lambda[\cV_\Gamma]_0\to\Lambda[\cV_\Gamma]^0\),\\
\rH^1_\sing(K,\rH^1(X_{\overline{K}},\Lambda(1)))&=
\coker\(\gamma\circ\rho\colon\Lambda[\cV_\Gamma]_0\to\Lambda[\cV_\Gamma]^0\).
\end{align*}
\end{example}

\section{Singular quotient via Galois deformation}\label{S: singular}

Under the assumption of Proposition \ref{pr:level_raising}, we obtain the diagram (with the dashed line in question)
\begin{align}\label{eq:singular_1}
\xymatrix{
\(\dfrac{\rE^{1,2r-2}_{0,2}(-1)}{\mu\rE^{-1,2r}_{0,2}}\)/\fn_0^m \ar[d]\ar[r] &
\rH^1_\sing(\dQ_{p^2},\rH^{2r-1}(\rM'_0\otimes_{\dQ_{p^2}}\ol\dQ_p,\dZ_\lambda(r)))/\fn_0^m \ar[d] \\
\dfrac{\dZ_\lambda[\rS_0^\circ]}{\fn_0^m} \ar@{-->}[r]^-{?} & \rH^1_\sing(\dQ_{p^2},\rH^{2r-1}(\rM'_0\otimes_{\dQ_{p^2}}\ol\dQ_p,\dZ_\lambda(r))/\fn_0^m)
}
\end{align}
by combining with \eqref{eq:singular}. In this section, we explain when all arrows in the above diagram are isomorphisms in the case where $\Pi=\Pi^A$ (hence $\dQ(\Pi)=\dQ$ and $\lambda=\ell$).

Recall that $\rM'_\alpha\simeq\cS_\alpha\otimes_F\dQ_{p^2}$ from Proposition \ref{pr:moduli}(4). We propose one more condition on $\ell$:
\begin{description}
  \item[(L4)] For $\alpha=0,1$, $\rH^i(\widetilde\cS_\alpha\otimes_F\ol\dQ,\dZ_\ell)_{\fn_\alpha}=0$ for $i\neq n_\alpha-1$ and $\rH^{n_\alpha-1}(\widetilde\cS_\alpha\otimes_F\ol\dQ,\dZ_\ell)_{\fn_\alpha}$ is a free $\dZ_\ell$-module of finite rank. Here, $\widetilde\cS_\alpha$ corresponds to an arbitrary open compact subgroup $\widetilde{K}_\alpha\subseteq K_\alpha$ satisfying $\widetilde{K}_\alpha^\Sigma=K_\alpha^\Sigma$.
\end{description}

\begin{proposition}\label{pr:singular}
Let $A=(A_0,A_1)$ be a pair of rational elliptic curves satisfying Theorem \ref{th:sym}(c,d) and assume $F^+/\dQ$ solvable. There exists a positive integer $\ell^A$ depending only on $A$ such that for every prime number $\ell\geq \ell^A$ satisfying (L1--4), if $p$ is a level-raising prime modulo $\ell^m$ (Definition \ref{de:level_raising}) satisfying the extra condition:
\begin{itemize}
  \item[$(*)$] The natural quotient map
     \[
     \frac{\dZ_\ell[\rS_0^\circ]}{\fn_0^m}\to\frac{(\dZ_\ell/\ell^m)[\rS_0^\circ]}{\Ker\phi_0^\Pi}
     \]
     is an isomorphism of nontrivial $\dZ_\ell$-modules.
\end{itemize}
then
\begin{itemize}
  \item all maps in \eqref{eq:singular_1} are isomorphism; and

  \item $\rH^{2r-1}(\cS_0\otimes_F\ol\dQ,\dZ_\ell(r))/\fn_0^m$ is isomorphic to $(\rT^A_{0,\ell}\otimes\dZ/\ell^m)^{\oplus\mu}$ for some $\mu>0$ as a $\dZ_\ell[\Gal(\ol\dQ/F)]$-module.
\end{itemize}
\end{proposition}

Of course, the whole statement only depends on the factor $A_0$ of the pair. However, we still write in this slightly redundant way for the sake of unity in notation.

The proof of Proposition \ref{pr:singular} uses the theory of Galois deformation (of the residual representation $\rT^A_{0,\ell}\otimes\dF_\ell$). We fix a prime number $p$ as in the statement and start from the local Galois deformation at the place $\fp$ of $F$. Let $\rI_\fp$ be the inertia subgroup of the decomposition subgroup $\rG_\fp\coloneqq\Gal(\ol\dQ_p/F_\fp)\subseteq\Gal(\ol\dQ/F)$. Fix an arithmetic Frobenius element $\phi_\fp$ of $\rG_\fp$ and a topological generator $t$ of the pro-$\ell$ tame inertia group at $\fp$. By (P3, P4), for every Artinian local $\dZ_\ell$-algebra $(R,\fm_R)$ with $R/\fm_R=\dF_\ell$ and every conjugate self-dual lifting $\rT$ of $\rT^A_{0,\ell}\otimes\dF_\ell$ as an $R[\rG_\fp]$-module, there is a unique decomposition $\rT=\rT^\dag\oplus\rT^\ddag$ into free $R$-modules such that the characteristic polynomial of the action of $\phi_\fp$ on $\rT^\dag$ equals $(T-1)(T-p^2)$ modulo $\fm_R$. Define
\begin{itemize}
  \item $\sD_\fp^\mix$ to be the local deformation problem at $\fp$ parameterizing $\rT$ such that $\rI_\fp$ preserves the decomposition $\rT=\rT^\dag\oplus\rT^\ddag$ and acts trivially on $\rT^\ddag$;

  \item $\sD_\fp^\ram$ to be the local deformation problem at $\fp$ contained in $\sD_\fp^\mix$ such that the characteristic polynomial of the action of $\phi_\fp$ on $\rT^\dag$ equals $(T-1)(T-p^2)$;

  \item $\sD_\fp^\unr$ to be the local deformation problem at $\fp$ contained in $\sD_\fp^\mix$ on which the action of $\rI_\fp$ is trivial.
\end{itemize}
Globally, for $?=\mix,\ram,\unr$, we define $\sfR^?$ to be the local $\dZ_\ell$-algebra that parameterizes conjugate self-dual liftings of $\rT^A_{0,\ell}\otimes\dF_\ell$ as $\Gal(\ol\dQ/F)$-modules that are unramified at primes not above $\Sigma\cup\{\ell\}$ and not $\fp$, are Fontaine--Laffaille at primes above $\ell$, and belong to $\sD_\fp^?$ at $\fp$. Denote by $\rT^?_\univ$ the universal lifting on a free $\sfR^?$-module of rank $n_0$. Finally, put
\[
\sfR^{\r{cong}}\coloneqq\sfR^\unr\otimes_{\sfR^\mix}\sfR^\ram.
\]

We fix eigenvectors $\bar\sfv$ and $\bar\sfv'$ in $\rT^A_{0,\ell}\otimes\dF_\ell$ for $\phi_\fp$ with the eigenvalues $1$ and $p^2$, respectively. By Hensel's lemma, $\bar\sfv$ and $\bar\sfv'$ lift uniquely to eigenvectors in $\rT^\unr_\univ$ for $\phi_\fp$, whose eigenvalues we denote by $\sfs$ and $\sfs'$, respectively, which are elements of $\sfR^\mix$. Moreover, we may write $t\cdot\sfv=\sfx\sfv'+\sfv$ for a unique element $\sfx\in\sfR^\mix$. Then we have $\sfx(\sfs-1)=0$ and
\[
\sfR^\unr=\sfR^\mix/(\sfx),\quad
\sfR^\ram=\sfR^\mix/(\sfs-1),\quad
\sfR^{\r{cong}}=\sfR^\mix/(\sfs-1,\sfx).
\]

Let $\sfT^\unr$ be the image of $\dT^\Sigma_0$ in $\End_{\dZ_\ell}(\dZ_\ell[\rS_0^\circ])$, so that $\sfT^\unr_{\fn_0}$ is nonzero by $(*)$. We show in \cite{LTXZZ2} that there is a canonical homomorphism $\sfR^\unr\to\sfT^\unr_{n_0}$. Moreover, there exists a spherical Hecke operator $\tN$ at $\fp^+$ belonging to $\fn_0^m$ such that its action on $\dZ_\ell[\rS_0^\circ]_{\fn_0}$ coincides with $\sfs-1$;\footnote{Indeed, $\tN$ is the same operator in the proof of Proposition \ref{pr:level_raising}.} in particular, we have
\begin{align}\label{eq:deformation_1}
\dZ_\ell[\rS_0^\circ]_{\fn_0}\otimes_{\sfR^\unr}\sfR^{\r{cong}}=
\frac{\dZ_\ell[\rS_0^\circ]_{\fn_0}}{\tN\dZ_\ell[\rS_0^\circ]_{\fn_0}}.
\end{align}

On the other hand, let $\sfT^\ram$ be the image of $\dT^{\Sigma\cup\{p\}}_0$ in $\End_{\dZ_\ell}(\rH^{2r-1}(\rM'_0\otimes_{\dQ_{p^2}}\ol\dQ_p,\dZ_\ell(r)))$. By $(*)$ and Proposition \ref{pr:level_raising}, $\sfT^\ram$ is nonzero as well. We show in \cite{LTXZZ2} that there is a canonical homomorphism $\sfR^\ram\to\sfT^\ram_{n_0}$. Put
\[
\sfH\coloneqq\Hom_{\sfR^\ram[\Gal(\ol\dQ/F)]}
\(\rT^\ram_\univ,\rH^{2r-1}(\rM'_0\otimes_{\dQ_{p^2}}\ol\dQ_p,\dZ_\ell(r))_{\fn_0}\),
\]
which is naturally an $\sfR^\ram$-module. Suppose that the tautological map
\begin{align}\label{eq:deformation}
\sfH\otimes_{\sfR^\ram}\rT^\ram_\univ\to\rH^{2r-1}(\rM'_0\otimes_{\dQ_{p^2}}\ol\dQ_p,\dZ_\ell(r))_{\fn_0}
\end{align}
is an isomorphism. Then we have
\begin{align}\label{eq:deformation_2}
&\rH^1_\sing(\dQ_{p^2},\rH^{2r-1}(\rM'_0\otimes_{\dQ_{p^2}}\ol\dQ_p,\dZ_\ell(r))_{\fn_0})
=\sfH\otimes_{\sfR^\ram}\rH^1_\sing(\dQ_{p^2},\rT^\ram_\univ) \\
&\quad=\sfH\otimes_{\sfR^\ram}(\sfR^\ram\sfv/\sfx\sfv)
\simeq\sfH\otimes_{\sfR^\ram}\sfR^\ram/(\sfx)
=\sfH\otimes_{\sfR^\ram}\sfR^{\r{cong}}.\notag
\end{align}
By (L4) and \cite{LTXZZ2}*{Theorem~3.6.3}, there exists a positive integer $\ell^A$ such that for every prime number $\ell\geq \ell^A$ satisfying (L1--4), the following holds:
\begin{itemize}
  \item The homomorphism $\sfR^\unr\to\sfT^\unr_{n_0}$ is an isomorphism; and $\dZ_\ell[\rS_0^\circ]_{\fn_0}$ is a free $\sfR^\unr$-module of finite rank.

  \item The homomorphism $\sfR^\ram\to\sfT^\ram_{n_0}$ is an isomorphism; both $\rH^{2r-1}(\rM'_0\otimes_{\dQ_{p^2}}\ol\dQ_p,\dZ_\ell(r))_{\fn_0}$ and $\sfH$ are free $\sfR^\ram$-module of finite ranks; and the tautological map \eqref{eq:deformation} is an isomorphism.
\end{itemize}
Let $d^\unr$ and $d^\ram$ be the ranks of $\dZ_\ell[\rS_0^\circ]_{\fn_0}$ and $\sfH$ over $\sfR^\unr$ and $\sfR^\ram$, respectively.

\begin{proof}[Proof of Proposition \ref{pr:singular}]
Let $\ell$ be a prime number that is larger than $\ell^A$ above and satisfies (L1--4). Indeed, as pointed out in the proof of Proposition \ref{pr:level_raising}, the map
\[
\(\frac{\rE^{1,2r-2}_{0,2}(-1)}{\mu\rE^{-1,2r}_{0,2}}\)_{\fn_0}
\to\frac{\dZ_\ell[\rS_0^\circ]_{\fn_0}}{\tN\dZ_\ell[\rS_0^\circ]_{\fn_0}}
\]
is already surjective. On the other hand, we have an injective map
\[
\(\frac{\rE^{1,2r-2}_{0,2}(-1)}{\mu\rE^{-1,2r}_{0,2}}\)_{\fn_0}
\to\rH^1_\sing(\dQ_{p^2},\rH^{2r-1}(\rM'_0\otimes_{\dQ_{p^2}}\ol\dQ_p,\dZ_\ell(r))_{\fn_0})
\]
from \eqref{eq:singular}. Thus, by \eqref{eq:deformation_1} and \eqref{eq:deformation_2}, the proposition can be reduced to showing that $d^\unr=d^\ram$ since $\sfR^{\r{cong}}$ is a finite ring. To compare $d^\unr$ and $d^\ram$, we specialize both modules to a complex point and compare them using automorphic methods, whose details can be found in the proof of \cite{LTXZZ}*{Proposition~6.4.1}.
\end{proof}

\begin{remark}
When we compare $d^\unr$ and $d^\ram$ in the proof of \cite{LTXZZ}*{Proposition~6.4.1}, we require $K_v$ to be small enough (or \emph{transferable} \cite{LTXZZ}*{Definition~D.2.1} to be more precise) for $v$ above $\Sigma$. In this remark, we point out that this is unnecessary. Indeed, for every $v$ above $\Sigma$, we can find a normal subgroup $\widetilde{K}_v\subseteq K_v$ of finite index such that $\widetilde{K}_v$ is transferable; put $\widetilde{K}\coloneqq\(\prod_{v\mid\Sigma}\widetilde{K}_v\)K^\Sigma$. The Hoschchild--Serre spectral sequence together with (L4) imply that
\[
\rH^{2r-1}(\cS_0\otimes_F\ol\dQ,\dZ_\ell(r))_{\fn_0}=
\rH^{2r-1}(\widetilde\cS_0\otimes_F\ol\dQ,\dZ_\ell(r))_{\fn_0}^{K/\widetilde{K}}.
\]
Thus, the relation $d^\unr=d^\ram$ for $\widetilde{K}$ implies that for $K$.
\end{remark}

\section{Explicit reciprocity law}
\label{ss:8}

In this section, we combine the arithmetic input in the previous two sections to deduce an explicit reciprocity law for the diagonal cycle in the product $\rM'_0\times\rM'_1$. Together with the automorphic input, we show that a system of Selmer annihilators for $\rT^A_\ell$ (Definition \ref{de:annihilator}) exists for all but finitely many prime numbers $\ell$ in the situation of Theorem \ref{th:sym}, hence finish the proof of that theorem. Recall that we have assumed $F^+\neq\dQ$ and $n=2r$ even.

Take prime numbers $\ell\geq \ell^A$ satisfying (L1--4) and a level-raising prime $p$ modulo $\ell^m$ (Definition \ref{de:level_raising}) satisfying:
\begin{itemize}
  \item[(P6)] For $\alpha=0,1$, the natural quotient map
     \[
     \frac{\dZ_\ell[\rS_\alpha^\circ]}{\fn_\alpha^m}\to
     \frac{(\dZ_\ell/\ell^m)[\rS_\alpha^\circ]}{\Ker\phi_\alpha^\Pi}
     \]
     is an isomorphism of nontrivial $\dZ_\ell$-modules.
\end{itemize}

Put $\bP\coloneqq\bM_0\otimes_{\dZ_{p^2}}\bM_1$, $\rP'\coloneqq\bP\otimes_{\dZ_{p^2}}\dQ_{p^2}$, and $\rP\coloneqq\bP\otimes_{\dZ_{p^2}}\dF_{p^2}$. Let $\bD$ be the graph of the morphism $\bm\colon\bM_0\to\bM_1$, which is a closed subscheme of $\bP$ of codimension $n$; similarly, we have $\rD'$ and $\rD$. Put $\cP\coloneqq\cS_0\otimes_F\cS_1$ and let $\cD$ be the graph of the morphism $\sigma$ (see Proposition \ref{pr:moduli}(4)) so that the base change of the pair $(\cP,\cD)$ from $F$ to $F_\fp=\dQ_{p^2}$ coincides with $(\rP',\rD')$.

Put $\fn^m\coloneqq\fn_0^m\times\fn_1^m$. Then (L4) implies that $\rH^i(\cP\otimes_F\ol\dQ,\dZ_\ell)_{\fn}=0$ for $i\neq 2n-1$. In particular, we have an induced class
\[
\alpha(\cD)\in\rH^1(F,\rH^{2n-1}(\cP\otimes_F\ol\dQ,\dZ_\ell(n))/\fn^m)
\]
via the Hoschchild--Serre spectral sequence. Denote by
\[
\partial_\fp\colon\rH^1(F_\fp,-)\to\rH^1_\sing(F_\fp,-)=\rH^1_\sing(\dQ_{p^2},-)
\]
the singular quotient map. Then we have the element
\[
\partial_\fp(\loc_\fp(\alpha(\cD)))\in
\rH^1_\sing(\dQ_{p^2},\rH^{2n-1}(\rP'\otimes_{\dQ_{p^2}}\ol\dQ_p,\dZ_\ell(n))/\fn^m).
\]

\begin{proposition}[First explicit reciprocity law]\label{pr:reciprocity}
Under the situation above, the following holds.
\begin{enumerate}
  \item The $\dZ/\ell^m[\Gal(\ol\dQ/F)]$-module $\rH^{2n-1}(\cP\otimes_F\ol\dQ,\dZ_\ell(n))/\fn^m$ is nonzero and isomorphic to the direct sum of finitely many modules of the form $\rT^A_\ell\otimes\dZ/\ell^{m'}$ with $m'\leq m$.

  \item The natural map
     \begin{align*}
     &\rH^1_\sing(\dQ_{p^2},\rH^{2r-1}(\rM'_0\otimes_{\dQ_{p^2}}\ol\dQ_p,\dZ_\ell(r))/\fn_0^m)
     \otimes_{\dZ_\ell}\rH^0(\dQ_{p^2},\rH^{2r}(\rM'_1\otimes_{\dQ_{p^2}}\ol\dQ_p,\dZ_\ell(r))/\fn_1^m) \\
     &\quad\to\rH^1_\sing(\dQ_{p^2},\rH^{2n-1}(\rP'\otimes_{\dQ_{p^2}}\ol\dQ_p,\dZ_\ell(n))/\fn^m)
     \end{align*}
     is an isomorphism. In particular, we have a natural isomorphism
     \[
     \nabla_{/\fn^m}\colon\rH^1_\sing(\dQ_{p^2},\rH^{2n-1}(\rP'\otimes_{\dQ_{p^2}}\ol\dQ_p,\dZ_\ell(n))/\fn^m)
     \xrightarrow\sim\frac{(\dZ_\ell/\ell^m)[\rS_0^\circ]}{\Ker\phi_0^\Pi}\otimes
     \frac{(\dZ_\ell/\ell^m)[\rS_1^\circ]}{\Ker\phi_1^\Pi}
     \]
     from Proposition \ref{pr:tate}, Proposition \ref{pr:singular}, and (P6).

  \item Under the natural pairing
     \[
     \frac{(\dZ_\ell/\ell^m)[\rS_0^\circ]}{\Ker\phi_0^\Pi}\otimes
     \frac{(\dZ_\ell/\ell^m)[\rS_1^\circ]}{\Ker\phi_1^\Pi}\times
     (\dZ_\ell/\ell^m)[\rS_0^\circ][\phi_0^\Pi]\otimes(\dZ_\ell/\ell^m)[\rS_1^\circ][\phi_1^\Pi]
     \to\dZ/\ell^m,
     \]
     the paring between $\nabla_{/\fn^m}(\partial_\fp(\loc_\fp(\alpha(\cD))))$ and every pure tensor $f_0\otimes f_1\in(\dZ_\ell/\ell^m)[\rS_0^\circ][\phi_0^\Pi]\otimes(\dZ_\ell/\ell^m)[\rS_1^\circ][\phi_1^\Pi]$ is equal to
     \[
     \gamma^A_p\cdot\sum_{\rU(V)(F^+)\backslash\rU(V)(\dA^\infty_{F^+})/K}f_0(h)f_1(h),
     \]
     where $\gamma^A_p$ is an element in $\dZ_{(\ell)}^\times$ depending only on $p$ and $A$ (or rather $a_p(A_0)$ and $a_p(A_1)$).
\end{enumerate}
\end{proposition}

\begin{proof}
For (1), by \cite{LTXZZ}*{Proposition~3.2.11}, we have an isomorphism
\[
\rH^n(\cS_1\otimes_F\ol\dQ,\dQ_\ell(r))/\phi^\Pi_1\simeq(\rT^A_{1,\ell}\otimes_{\dZ_\ell}\dQ_\ell)^{\oplus\mu}
\]
of $\dQ_\ell[\Gal(\ol\dQ/F)]$-modules for some integer $\mu\geq 0$.

By (L3), $\rT^A_{1,\ell}\otimes\dF_\ell$ is absolutely irreducible, which implies that $\rH^n(\cS_1\otimes_F\ol\dQ,\dQ_\ell(r))/\fn^m_1$ is isomorphic to the direct sum of finitely many modules of the form $\rT^A_\ell\otimes\dZ/\ell^{m'}$ with $m'\leq m$. By Proposition \ref{pr:tate} and (P6), we know that $\rH^n(\cS_1\otimes_F\ol\dQ,\dQ_\ell(r))/\fn^m_1$ is nonzero. Then (1) follows from (L4), Proposition \ref{pr:singular}, and the K\"{u}nneth formula.

Part (2) is deduced in the proof of \cite{LTXZZ}*{Theorem~7.2.8}. Part (3) is simply \cite{LTXZZ}*{Theorem~7.2.8(3)}, where one can find the explicit formula for $\gamma^A_p$.
\end{proof}

\begin{corollary}
Suppose that $L(0,\rV^A_F)$ does not vanish. Then for every prime number $\ell\geq \ell^A$ satisfying (L1--4), a system of Selmer annihilators for $\rT^A_\ell$ (Definition \ref{de:annihilator}) exists.
\end{corollary}

\begin{proof}
Choose a standard definite hermitian space $V$, a $\Sigma$-level pair $(K,K^\sharp)$ and elements $(f_0,f_1)$ as in Proposition \ref{pr:ggp}. Without lost of generality, we may assume that neither $f_0$ nor $f_1$ is divisible by $\ell$. Let $m_\ell\geq 0$ be the $\ell$-valuation of the sum
\[
\sum_{\rU(V)(F^+)\backslash\rU(V)(\dA^\infty_{F^+})/K}f_0(h)f_1(h).
\]
For every fixed integer $m>m_\ell$, it is easy to see that all but finitely many level-raising prime $p$ modulo $\ell^m$ satisfy (P6). For every such prime number $p$, fix an isomorphism
\[
\rH^1_\sing(F_\fp,\rT^A_\ell\otimes\dZ/\ell^m)\simeq\dZ/\ell^m.
\]
Then by Proposition \ref{pr:reciprocity}(1), Proposition (2) induces an isomorphism
\[
\ts\colon
(\dZ_\ell/\ell^m)[\rS_0^\circ][\phi_0^\Pi]\otimes(\dZ_\ell/\ell^m)[\rS_1^\circ][\phi_1^\Pi]
\xrightarrow\sim
\Hom_{\dZ/\ell^m[\Gal(\ol\dQ/F)]}\(\rH^{2n-1}(\cP\otimes_F\ol\dQ,\dZ_\ell(n))/\fn^m,\rT^A_\ell\otimes\dZ/\ell^m\).
\]
Put
\[
c_{m,p}\coloneqq\rH^1(F,\ts(f_0\otimes f_1))\alpha(D)\in\rH^1(F,\rT^A_\ell\otimes\dZ/\ell^m).
\]
Then by construction, $c_{m,p}$ satisfies (a) and (c) in Definition \ref{de:annihilator}; and by Lemma \ref{le:local_galois} and Proposition \ref{pr:reciprocity}(3), it also satisfies (b) in Definition \ref{de:annihilator}. The corollary follows.
\end{proof}

By \cite{LTXZZ}*{Corollary~D.1.4} (which is essentially a consequence of \cite{CS17}), all but finitely many prime numbers $\ell$ satisfy (L4). Thus, in view of Proposition \ref{pr:kolyvagin}, Theorem \ref{th:sym} is now proved.

\section{Relation to categorical local Langlands}

The recently emerging categorical local Langlands correspondence predicts the existence of certain coherent sheaves on the moduli stack of local Langlands parameters of representation theoretic meaning. Among other things, ($!$-)pullbacks of such sheaves to the stack of global Langlands parameters govern the (integral) cohomology of Shimura varieties and their specific subvarieties. This can be viewed as a refinement of local-global compatibility. We refer to \cite{Zhu20}*{\S4.7} for preliminary discussions and to the forthcoming work \cite{EZ} for a
precise formulation. As explained in \cite{Zhu20}, the work of \cite{XZ} on the geometric realization of certain Jacquet-Langlands correspondences fits into this framework. Here we explain that various maps of Hecke modules in this note can also be interpreted within this framework.

In what follows, for a closed immersion $Y\hookrightarrow X$ of derived stacks (the most general ``space'' we will encounter in this section) and a sheaf $F$ on $Y$, we will regard it as a sheaf on $X$ by extension of zero according to the context; and for a ring $R$, write $\ul{R}_X$ for the constant sheaf valued in $R$ on $X$ in suitable sense.

We fix a prime number $\ell$ as before. Choose an odd prime $p\neq \ell$ and put ourselves in the initial setup of Section \ref{ss:4} but without assuming that $n$ is even (but still at least $2$). In particular, $F^+_{\fp^+}=\dQ_p$, $F_\fp=\dQ_{p^2}$, and $\tc$ gives the generator of $\Gal(\dQ_{p^2}/\dQ_p)$. Write $n=2r+\alpha$ with $r\in\dZ$ and $\alpha\in\{0,1\}$. In this section, we only work with one hermitian space, namely, $V$. Thus, we will suppress the parity $\alpha$ in the scripts of all notations (such as $\bM_\alpha,\rM_\alpha^?,\rB_\alpha^?,\rS_\alpha^?,\dots$). Put $\sfV\coloneqq V\otimes_{F^+}F^+_{\fp^+}$ and $\sfG\coloneqq\rU(\sfV)$. Let $\sfK$ be the $\fp^+$-component of $K$, which is the stabilizer of a self-dual lattice of $\sfV$ hence a hyperspecial maximal subgroup of $\sfG(\dQ_p)$. Let $(V',K')$ be the pair that appears in Proposition \ref{pr:moduli}(4). Put $\sfV'\coloneqq V'\otimes_{F^+}F^+_{\fp^+}$ and $\sfG'\coloneqq\rU(\sfV')$. Let $\sfK'$ be the $\fp^+$-component of $K'$, which is the stabilizer of an almost self-dual lattice of $\sfV'$ hence a special maximal subgroup of $\sfG'(\dQ_p)$.

Recall that we have the moduli scheme $\bM$ from Proposition \ref{pr:moduli}, which is a projective strictly semistable scheme over $\Spec\dZ_{p^2}$ of pure relative dimension $n-1$. Consider the nearby cycles $\rR\Psi(\dZ_\ell[n-1])$ on $\rM$, which fits into the short exact sequence
\begin{align}\label{eq:zhu1}
0 \to \rR j^\bullet_!\ul{\dZ_\ell}_{\rM^\bullet\setminus\rM^\dag}[n-1]\to\rR\Psi(\dZ_\ell[n-1])
\to \rR j^\circ_*\ul{\dZ_\ell}_{\rM^\circ\setminus\rM^\dag}[n-1] \to 0
\end{align}
of perverse sheaves on $\rM$. Here, $j^\circ\colon\rM^\circ\setminus\rM^\dag\to\rM^\circ$ and $j^\bullet\colon\rM^\bullet\setminus\rM^\dag\to\rM^\bullet$ denote the two open immersions. Furthermore, we have the short exact sequence
\begin{align}\label{eq:zhu2}
0 \to \ul{\dZ_\ell}_{\rM^\circ}[n-1] \to \rR j^\circ_*\ul{\dZ_\ell}_{\rM^\circ\setminus\rM^\dag}[n-1]
\to \ul{\dZ_\ell}_{\rM^\dag}[n-2](-1) \to 0
\end{align}
of perverse sheaves on $\rM^\circ$, and similarly the short exact sequence
\begin{align}\label{eq:zhu3}
0 \to \ul{\dZ_\ell}_{\rM^\dag}[n-2] \to \rR j^\circ_!\ul{\dZ_\ell}_{\rM^\bullet\setminus\rM^\dag}[n-1]
\to \ul{\dZ_\ell}_{\rM^\bullet}[n-1] \to 0
\end{align}
of perverse sheaves on $\rM^\bullet$. Note that \eqref{eq:zhu1}, \eqref{eq:zhu2} and \eqref{eq:zhu3} together give the usual weight filtration $\rW_0\subseteq\rW_1\subseteq\rW_2=\rR\Psi(\dZ_\ell[n-1])$ such that
\begin{align*}
\gr_0^\rW&\coloneqq\rW_0=\ul{\dZ_\ell}_{\rM^\dag}[n-2],\\
\gr_1^\rW&\coloneqq\rW_1/\rW_0=\ul{\dZ_\ell}_{\rM^\circ}[n-1]\oplus\ul{\dZ_\ell}_{\rM^\bullet}[n-1],\\
\gr_2^\rW&\coloneqq\rW_2/\rW_1=\ul{\dZ_\ell}_{\rM^\dag}[n-2](-1).
\end{align*}
The weight spectral sequence $\rE^{p,q}_s$ (for $n$) appearing in Section \ref{ss:4} is obtained by taking geometric cohomology of the above graded terms, up to certain Tate twist and degree shifts.

All the above (perverse) sheaves descend to the moduli stack of local Shtukas for appropriate parahoric integral model of $\sfG$ over $\dZ_{p^2}$, and then have further pushforward to the moduli stack of $\sfG$-isocrystals. Similar observation applies to constant sheaves (in $\dZ_\ell$) on $\rS^\circ,\rS^\bullet,\rS^\dag$, together with their Tate twists and degree shifts. 

To explain this, let us set up a few notations. Fix an element $\fj\in\dZ_{p^2}^\times$ satisfying $\fj+\fj^\tc=0$. We choose a basis $\{e_1,\ldots, e_n\}$ of $\sfV$ under which the hermitian form is given by
\begin{equation}\label{eq:Jn}
\rJ_n\coloneqq \fj^{n-1}\cdot\begin{pmatrix}
 & &&& 1 \\
 &&& -1 & \\
 && 1 && \\
 & -1 &&& \\
 \iddots &&&&
\end{pmatrix}
\end{equation}
(while $\rJ_1=1$). We also identify (the underlying vector spaces of) $\sfV'$ and $\sfV$ and identify $\sfG'$ with the unitary group preserving
\[
\rJ'_n\coloneqq b\cdot \rJ_n=\begin{pmatrix} -\rJ_{n-1} &\\ & p  \end{pmatrix},\quad \mbox{where}\quad b\coloneqq\begin{pmatrix} &\fj^{-1}\sfI_{n-1}  \\  p \fj^{1-n}& \end{pmatrix}.
\]
In this way, we identify $\sfG'$ as an extended pure inner form of $\sfG$, that is, $\sfG'=\sfG_b$ is the $\sigma$-centralizer of $b\in\sfG(\dQ_{p^2})=\GL(\sfV)(\dQ_{p^2})$, where $\sigma$ is the Frobenius endomorphism of $\sfG(\dQ_{p^\infty})$ (and $\dQ_{p^\infty}$ denotes a fixed complete maximal unramified extension of $\dQ_{p^2}$).

For $i=0,1,\ldots,n-1$, we have a $\dZ_{p^2}$-lattice
\[
\Lambda_i=\sum_{j\leq i} p^{-1} \dZ_{p^2}\cdot e_j +  \sum_{j>i} \dZ_{p^2}\cdot e_j
\]
in $\sfV$, so that $\Lambda_0\subsetneq \Lambda_1\subsetneq \cdots$ is a standard $\dZ_{p^2}$-lattice chain in $\sfV$. As usual, we extend it to a periodic lattice chain by setting $\Lambda_{i+n}=p^{-1}\Lambda_i$. For every subset $I\subset \{0,1,\ldots, n-1\}$, let $\sG_I$ be the parahoric group scheme of $\sfG_{\dQ_{p^2}}=\sfG'_{\dQ_{p^2}}=\GL(\sfV)$ over $\dZ_{p^2}$ preserving the corresponding lattice chain. Note that if we identify $\{0,1,\ldots, n-1\}$ with the set of vertices of the standard alcove in the Bruhat--Tits building of $\GL(\sfV)$, then the Frobenius of $\sfG$ acts on this set by sending $i$ to $n-i$ while the Frobenius of $\sfG'$ acts on this set by sending $i$ to $n-1-i$. Therefore, $\sG_I$ descends to a parahoric group scheme of $\sfG$ over $\dZ_p$ if and only if $i\in I\iff n-i\in I$, and descends to a parahoric group scheme of $\sfG'$ over $\dZ_p$ if and only if $i\in I\iff n-1-i\in I$.

\begin{example}\label{ex:relevant parahoric}
The following parahoric group schemes have already been used previously.
\begin{itemize}
  \item $\sG_{\{0\}}$ is a hyperspecial parahoric of $\sfG$ and $\sG_{\{0\}}(\dZ_p)=\sfK$;

  \item $\sG_{\{0,n-1\}}$ is a special parahoric of $\sfG'$ and $\sG_{\{0,n-1\}}(\dZ_p)=\sfK'$;

  \item $\sG_{\{r, n-r\}}$ is a special parahoric of $\sfG$ (which is hyperspecial if $n=2r$ is even) and $\sG_{\{r,n-r\}}(\dZ_p)=\sfK^\bullet$;

  \item $\sG_{\{0,r,n-r\}}(\dZ_p)=\sfK^{\dag}$.
\end{itemize}
\end{example}

Now for each subset $I\subset \{0,1,\ldots, n-1\}$, let $\Sht_I^{\loc}$ be the moduli of $\sG_I$-Shtukas over $\Spec \mathbb F_{p^2}$. Let us recall its moduli interpretation following \cite{Zhu20}*{\S 3.1.4}. For a perfect $\mathbb F_{p^2}$-algebra $R$, let $W(R)$ denote its ring of Witt vectors, equipped with the automorphism $\sigma_R$ lifting the absolute Frobenius of $R$. Then $\Sht_I^{\loc}(R)$ classifies the groupoid of pairs $(\sP,\psi)$ consisting of a $\sG_I$-torsor $\sP$ on $\Spec W(R)$ and an isomorphism of $\sfG$-torsors $\psi\colon\sP\res_{\Spec W(R)[1/p]}\simeq \sigma_R^*\sP\res_{\Spec W(R)[1/p]}$. For example, when $I=\{0,n-1\}$, $\Sht_{\{0,n-1\}}^{\loc}(R)$ classifies quadruples $(\sE_{-1}\subset \sE_0, \psi_{-1},\psi_0)$, where
\begin{itemize}
  \item $\sE_{-1}\subset \sE_0$ is an inclusion of rank $n$ finite projective $W(R)$-modules such that $\sE_0/\sE_{-1}$ is an invertible $W(R)/p$-module;

  \item $\psi_0\colon\sE_0[\frac{1}{p}]\simeq \sigma_R^* \sE_{-1}^\vee[\frac{1}{p}]$ and $\psi_{-1}\colon\sE_{-1}[\frac{1}{p}]\simeq \sigma_R^*\sE_0^\vee[\frac{1}{p}]$ are two $W(R)[\frac{1}{p}]$-module isomorphisms making the diagram
      \[
      \xymatrix{
      \sE_{-1}[\frac{1}{p}]\ar^-\simeq[r]\ar^{\psi_{-1}}_-\simeq[d] & \sE_0[\frac{1}{p}]\ar^{\psi_0}_-\simeq[d]\\
      \sigma_R^*\sE_0^\vee[\frac{1}{p}]\ar^-\simeq[r] & \sigma_R^* \sE_{-1}^\vee[\frac{1}{p}]
      }\]
      commutative.
\end{itemize}

We shall denote $\rM^{\loc}\subset\Sht_{\{0,n-1\}}^{\loc}$ the closed substack consisting of those quadruples such that both $\psi_{-1}$ and $\psi_0$ extend to $W(R)$-module maps and such that the cokernels of resulting maps are invertible $W(R)/p$-modules. Inside $\rM^{\loc}$ there are analogues local counterparts of $\rM^{\loc,\circ}$, $\rM^{\loc, \bullet}$, and $\rM^{\loc,\dag}$. Namely, $\rM^{\loc,\circ}$ consists of those $(\sE_{-1}\subset \sE_0,\phi_{-1},\phi_0)$ such that $\sE_0$ and $\sigma_R^*\sE_0^\vee$ coincide as submodules of $\sigma_R^*\sE_{-1}^\vee$, which is isomorphic to the perfection of $[\dP^{n-1}_{\Spec\dF_{p^2}}/\sfK]$. On the other hand, $\rM^{\loc,\bullet}$ consists of those $(\sE_{-1}\subset \sE_0,\psi_{-1},\psi_0)$ such that the cokernel $\sigma_R^*\sE_{-1}^\vee/\sE_{-1}$ is a finite projective $W(R)/p$-module of rank two. Finally, $\rM^{\loc, \dag}\coloneqq\rM^{\loc, \circ}\cap \rM^{\loc, \bullet}$ is isomorphic to the $\sfK$-quotient of the perfection of the Fermat hypersurface of degree $p+1$.

By abuse of notation, we still use $\rM$ to denote its own perfection. We have a morphism (sometimes called the crystalline period map)
\[
\loc_p\colon \rM\to \rM^{\loc}
\]
sending a quadruple in the Definition \ref{de:moduli} to $\sE_{-1}=(\sigma_R)_*\sD(A)_{\bar\tau}$, $\sE_0=(\sigma_R)_* \sD(A)_{\tau}^\vee$, where $\sD(A)$ denotes the covariant Dieudonn\'e module of the abelian scheme $A$ over $\Spec R$. The inclusion $\sE_{-1}\subset \sE_0$ is induced by the polarization $\lambda\colon A\to A^\vee$.

The pullback of $\rM^{\loc,?}\subset \rM^{\loc}$ under this morphism is the perfection of the corresponding scheme $\rM^{?}\subset \rM$. Then \eqref{eq:zhu1}, \eqref{eq:zhu2}, \eqref{eq:zhu3} are the $!$-pullback along $\loc_p$ of the corresponding cofiber sequences of $\dZ_\ell$-sheaves on $\rM^{\loc}$.




Let $\r{Isoc}_{\sfG}$ be the moduli of $\sfG$-isocrystals as defined in \cite{Zhu20}*{\S 3.1}. Recall that its $R$-points classify the groupoid of pairs $(\sP,\psi)$, where $\sP$ is a $\sfG$-torsor over $\Spec W(R)[1/p]$ that can be trivialized \'etale locally on $\Spec R$, and $\psi\colon\sP\simeq(\sigma_R)^*\sP$ is an isomorphism of $\sfG$-torsors. Let
\[
\mathrm{Nt}\colon \r{Sht}_I^\loc\to \r{Isoc}_{\sfG}\otimes \dF_{p^2}
\]
be the morphism sending $\sP$ to $\sP\res_{W(R)[1/p]}$. Then applying $\mathrm{Nt}_*$ to the local counterparts of \eqref{eq:zhu1}, \eqref{eq:zhu2}, \eqref{eq:zhu3}, we obtain corresponding cofiber sequences in the category of $\dZ_\ell$-sheaves on $\mathrm{Isoc}_G$.\footnote{We refer to \cite{Zhu25} for the definition of the category of $\ell$-adic sheaves on spaces such as $\r{Sht}^{\loc}_I$ and $\r{Isoc}_{\sfG}$.}




Next we describe the (conjectural) counterparts of those cofiber sequences on $\r{Isoc}_{\sfG}$ under the categorical local Langlands correspondence, which are cofiber sequences in the derived category of coherent complexes on the stack $\cX_\sfG$ over $\Spec\dZ_\ell$ that parameterizes tame local Langlands parameters of $\sfG$ with coefficients in $\dZ_\ell$. Various patched modules appearing in \cite{LTXZZ2} should be the completions of such coherent sheaves at appropriate points on $\cX_\sfG$. From now on, all schemes and stacks are understood in the derived sense, even if they are classical.

Denote by $\hat\sfG$ the dual group of $\sfG$ (over $\dZ$), which is nothing but $\GL_n$. The notion of $C$-group was first introduced in \cite{BG11}, with a different but equivalent formulation in \cite{Zhu24}. By \cite{Zhu24}*{\S1.1}, the $C$-group of $\sfG$ can be written as
\[
\pres{c}\sfG=(\hat\sfG\times\bG_m)\rtimes\Gal(\dQ_{p^2}/\dQ_p)
\]
in which the action of the unique generator of $\Gal(\dQ_{p^2}/\dQ_p)$ on $\hat\sfG\times\bG_m$ is given by the assignment $(A,\lambda)\mapsto(\lambda^{n-1}\fc(A),\lambda)$, where $\fc$ is the automorphism of $\GL_n$ given by
\begin{equation*}\label{eq: outer automorphism of GLn}
\fc: \GL_n\to\GL_n,\quad A\mapsto \rJ_n \pres{t}A^{-1}\rJ_n^{-1}
\end{equation*}
(see \eqref{eq:Jn} for $\rJ_n$).

In \cite{Zhu20}*{\S3.3}, the author constructed the stack of tame $L$-parameters $\r{Loc}^{\r{tame}}_{\pres{c}\sfG,\sfF^+}$ over $\Spec\dZ_\ell$ (essentially using \cite{LTXZZ2}*{Proposition~3.3.2}) and the affine scheme of framed $L$-parameters $\r{Loc}^{\r{tame},\Box}_{\pres{c}\sfG,\sfF^+}$. To simplify notation, put
\[
\cX_\sfG\coloneqq\r{Loc}^{\r{tame}}_{\pres{c}\sfG,\sfF^+}, \quad \cY_\sfG\coloneqq\r{Loc}^{\r{tame},\Box}_{\pres{c}\sfG,\sfF^+},
\]
so that $\cX_\sfG=\cY_\sfG/\hat\sfG_{\dZ_\ell}$.

\begin{remark}\label{rem: XG as fixed point stack}
We have the following remarks.
\begin{enumerate}
\item Consider the morphism $\fc_p\colon\hat\sfG\to\hat\sfG$ sending $A$ to $\fc(A^p)$, which is compatible with the $\fc$-conjugation action (of $\hat\sfG$ on itself) hence induces an isomorphism $\hat\sfG/\hat\sfG\to \hat\sfG/\hat\sfG$, still denoted by $\fc_p$. As explained in \cite{Zhu25}*{\S2.1.3}, after choosing a tame generator $\tau$ and a lifting of arithmetic Frobenius $\phi_p$ of $\dQ_p$, there is a natural isomorphism
    \[
    \cX_{\sfG}\simeq  \hat\sfG/\hat\sfG\times_{\r{id}\times\fc_p, (\hat\sfG/\hat\sfG\times  \hat\sfG/\hat\sfG), \Delta}  (\hat\sfG/\hat\sfG)_{\dZ_\ell}
    \]
    such that the restriction morphism $\cX_{\sfG}\to  \hat\sfG/\hat\sfG$ sending $\rho$ to $\rho(\tau)$ corresponds to the second projection on the right hand side.

\item The group $\pres{c}\sfG$ is slightly different from the group $\sG_n$ used in \cite{LTXZZ2}. However, as explained at the end of \cite{BG11}, the $\pres{c}\sfG$-valued $L$-parameters are equivalent to those $\sG_n$-valued Galois representations studied in \cite{LTXZZ2}.


\item We let $\fZ_{\sfG}$ denote the ring of regular functions on $\cX_{\sfG}$, which is called the tame spectral Bernstein center. Maximal ideals of $\fZ_{\sfG}$ can be identified with completely reducible tame $L$-parameters up to $\hat\sfG$-conjugacy, such that the canonical map $\cX_{\sfG}\to \Spec \fZ_{\sfG}$ sends an $L$-parameter $\rho$ to its semisimplification $\rho^{\r{ss}}$. The ring $\fZ_{\sfG}$ has a quotient ring $\fZ_{\sfG}^{\unr}$ classifying unramified completely reducible $L$-parameters up to conjugacy. Note that giving such an unramified $L$-parameter $\rho^{\r{ss}}$, we have $\rho^{\r{ss}}(\phi)=(A,p^{-2},1)\in \pres{c}\sfG$, where $\phi=\phi_p^2$ is (a lifting of) the arithmetic Frobenius of $\dQ_{p^2}$. Then up to $\hat\sfG$-conjugacy, $\rho^{\r{ss}}$ is uniquely determined by the set $\{\alpha_1,\ldots,\alpha_n\}$ of eigenvalues (with multiplicities) of the matrix $p^{1-n}A$, which form an (abstract) unitary Satake parameter in the sense of \cite{LTXZZ}*{Definition~3.1.3(1)}. Thus in the sequel, we will identify points of $\Spec \fZ_{\sfG}^{\unr}$ with unitary Satake parameters.
\end{enumerate}
\end{remark}

Recall that unipotent conjugacy classes of $\GL_n$ (over a field) are in bijection with partitions of $n$, and every unipotent conjugacy class $\sfO$ is stable under any automorphism of $\GL_n$ and therefore is $\Gal(\dQ_{p^2}/\dQ_p)$-stable.
For every unipotent conjugacy class $\sfO$ in $\hat\sfG_{\dQ_\ell}$, we denote by $\cX_{\sfG,\sfO}\subset \cX_{\sfG}$ the Zariski closure of the set of points $\rho\in\cX_\sfG(\ol\dQ_\ell)$ such that $\rho$ maps the tame inertia into $\sfO(\ol\dQ_\ell)$, equipped with the reduced closed subscheme structure. Then each member $\cX_{\sfG,\sfO}$ is an irreducible component of $\cX_\sfG$, flat over $\dZ_\ell$. For example, when $\sfO=1$ is the trivial orbit, $\cX_\sfG^\unr\coloneqq\cX_{\sfG,1}$ is the stack of unramified Langlands parameter; when $\sfO$ is the regular unipotent orbit, $\cX_{\sfG,\sfO}$ is the Steinberg component; when $\sfO$ is the unique nontrivial minimal unipotent orbit, denoted by $\sfO_\mnm$, we put $\cX_\sfG^\mnm\coloneqq\cX_{\sfG,\sfO_\mnm}$. Finally, put $\cX_{\sfG}^{\mix}\coloneqq\cX_\sfG^\unr\cup \cX_\sfG^\mnm$ and $\cX_\sfG^{\r{cong}}\coloneqq\cX_\sfG^\unr\cap \cX_\sfG^\mnm$ in $\cX_\sfG$.

Note that if $\rho\in \cX_{\sfG,\sfO}$ for some orbit $\sfO$, then $\rho^{\r{ss}}$ is unramified. More precisely, the morphism $\cX_{\sfG}\to \Spec\fZ_{\sfG}$ restricts to a morphism $\cX_{\sfG,\sfO}\to \Spec\fZ_{\sfG}^{\unr}$. One can show that there exists a unique quotient ring $\fZ_{\sfG}^{\sfO}$ of $\fZ_{\sfG}^{\unr}$ such that $\Spec \fZ_{\sfG}^{\sfO}\subset \Spec \fZ_{\sfG}^{\unr}$ is the scheme-theoretic image of the morphism $\cX_{\sfG,\sfO}\to \Spec\fZ_{\sfG}^{\unr}$.

\begin{example}\label{ex: Satake for min}
We have two examples.
\begin{enumerate}
  \item If $\sfO=1$, then $\fZ_{\sfG}^{\sfO}=\fZ_{\sfG}^{\unr}$, which is isomorphic to the ring of regular functions on $\cX_{\sfG}^{\unr}$.

  \item If $\sfO=\sfO_{\mnm}$, then $\Spec \fZ_{\sfG}^{\sfO}\subset \Spec \fZ_{\sfG}^{\unr}$ classifies those unitary Satake parameters containing the pair $\{(-1)^np,(-1)^np^{-1}\}$. If we identify $\fZ_{\sfG}^{\unr}$ with the spherical Hecke algebra of $\sfG$ via the Satake isomorphism (after choosing $\sqrt{p}$), then $\fZ_{\sfG}^{\sfO}=\fZ_{\sfG}^{\unr}/\tN\fZ_{\sfG}^{\unr}$ for some Hecke operator $\tN$. When $n$ is even, this Hecke operator is exactly the one in \eqref{eq:deformation_1}.
\end{enumerate}
\end{example}

\begin{remark}
Suppose that $\ell\nmid p^2-1$. Let $\rho$ be a geometric point of $\cX_{\sfG}^{\r{cong}}\otimes \dF_\ell$, with $\{\alpha_1,\ldots,\alpha_n\}$ the associated unitary Satake parameter. If $\rho$ is general,
If the pair $\{(-1)^np,(-1)^np^{-1}\}$ appears in $\{\alpha_1,\ldots,\alpha_n\}$ exactly once,\footnote{When $n$ is even, such unitary Satake parameter is called \emph{level-raising special} in \cite{LTXZZ}*{Definition~3.1.5}.} then both $\cX_{\sfG}^{\unr}$ and $\cX_{\sfG}^{\mnm}$ are smooth at this point, and the local equations around this point can be described similar to the Galois deformation rings in \S \ref{S: singular}. See \cite{LTXZZ2}*{Proposition~3.5.2} for details.
\end{remark}

The standard representation $\dZ^{\oplus n}$ of $\hat\sfG=\GL_n$, hence $\hat\sfG\times\bG_m$ by inflation, gives rise to a vector bundle on $\cX_\sfG$ (see the paragraph after \cite{Zhu20}*{Remark 2.2.15}), denoted by $\sS_\sfG$, equipped with a universal representation
\[
\rho_\univ\colon\Gal(\ol\dQ_p/\sfF)\to\GL_{\sO_{\cX_\sfG}}(\sS_\sfG)
\]
together with an isomorphism
\begin{align*}
\rho_\univ^\tc\simeq\rho_\univ^\vee(1-n).
\end{align*}
Here, $\rho_\univ^\tc$ denotes the conjugation of $\rho_\univ$ by $\tc$ (which can be canonically defined without choosing a lifting of the generator). In particular, we see that as vector bundles on $\cX_G$ there is a natural isomorphism $\sS_{\sfG}^\vee\simeq \sS_{\sfG}$.

\begin{remark}\label{rem: determinant of tautological representation}
As explained in \cite{Zhu20}*{\S 3.2}, $\cX_{\sfG}$ is a $\mu_2$-gerbe and therefore its category of coherent sheaves admits a $\dZ/2$-grading. Notice that $\sS_{\sfG}$ is in degree $1$, hence $\sD_{\sfG}\coloneqq\wedge^n\sS_{\sfG}$ is in degree $\alpha$. One also notices that when $n$ is even, $\sD_{\sfG}$ is non-trivial when restricted to a (stacky) point $\rho\in\cX_{\sfG}^{\unr}$ whose unitary Satake parameter contains $-1$, since the $n$th wedge of the standard representation of $\hat\sfG$ restricts to a non-trivial representation of the centralizer of such $\rho$. It follows that $\sD_{\sfG}$ is a line bundle of order two on $\cX_{\sfG}$. More precisely, keeping track of the Galois action, we have $\sD^{\otimes 2}_{\sfG}\simeq\sO_{\cX_{\sfG}}(n(n-1))$.
\end{remark}

Let $\hat\sfP\subset\hat\sfG$ be the standard maximal parabolic subgroup consisting of elements preserving the standard rank $(n-1)$-submodule $\dZ^{\oplus n-1}\subset \dZ^{\oplus n}$. Its unipotent radical $\sfU_{\hat\sfP}\subset \hat\sfP$ consists of those elements such that the induced automorphisms of $\dZ^{\oplus n-1}$ and $\dZ^{\oplus n}/\dZ^{\oplus n-1}$ are trivial.
We have characters
\[
\delta_{\hat\sfG,\hat\sfP}\colon\hat\sfP\to \hat\sfG\xrightarrow{\det}\bG_m,
\]
and
\[
\delta_{\hat\sfP}\colon
\hat\sfP\to \hat\sfP/\hat\sfU_{\hat\sfP}\to \GL_{n-1}\xrightarrow{\det}\bG_m.
\]
Let $\sO(\delta_{\hat\sfG,\hat\sfP})$ and $\sO(\delta_{\hat\sfP})$ denote the corresponding line bundles on $[\Spec \dZ/\hat\sfP]$ and their pullbacks to any stack equipped with a morphism to $[\Spec \dZ/\hat\sfP]$. For example, the canonical sheaf of $\hat\sfG/\hat\sfP\simeq \dP^{n-1}$ is $\sO(n(\delta_{\hat\sfP}-\delta_{\hat\sfG,\hat\sfP})+\delta_{\hat\sfG,\hat\sfP})$. Note that when forgetting the $\hat\sfG$-equivariant structure, the line bundle $\sO(\delta_{\hat\sfG,\hat\sfP}-\delta_{\hat\sfP})$ is the ample generator of $\hat\sfG/\hat\sfP$ and $\sO(\delta_{\hat\sfG,\hat\sfP})$ is trivial.

Let
\[
\r{Spr}_{\hat\sfP}\coloneqq\hat\sfG\times^{\hat\sfP}\sfU_{\hat\sfP}\to \hat{\sfG},\quad (g,u)\mapsto gug^{-1}
\]
be the partial Grothendieck--Springer resolution. The image of this map is the closure $\ol{\sfO_\mnm}$ of $\sfO_\mnm$ (in $\hat\sfG$).

Similarly, we have the standard maximal parabolic subgroup $\hat\sfQ\subset\hat\sfG$, consisting of elements preserving the standard rank one submodule $\dZ\subset \dZ^{\oplus n}$, with similarly defined characters $\delta_{\hat\sfG,\hat\sfQ}$, and the corresponding partial Grothendieck--Springer resolution $\r{Spr}_{\hat\sfQ}\coloneqq\hat\sfG\times^{\hat\sfQ}\sfU_{\hat\sfQ}\to \hat{\sfG}$, whose image is again $\ol{\sfO_\mnm}$. Put
\[
\r{St}_{\hat\sfP,\hat\sfQ}\coloneqq \sfU_{\hat\sfP}/\hat\sfP\times_{\hat\sfG/\hat\sfG} \sfU_{\hat\sfQ}/\hat\sfQ\simeq (\r{Spr}_{\hat\sfP}\times_{\hat{\sfG}} \r{Spr}_{\hat\sfQ})/\hat\sfG.
\]
We emphasize that we are taking the derived fiber product over $\hat{\sfG}$ (not over $\ol{\sfO_{\mnm}}$), and $\r{St}_{\hat\sfP,\hat\sfQ}$ indeed has non-trivial derived structure. However, its underlying classical substack $(\r{St}_{\hat\sfP,\hat\sfQ})_{\r{cl}}\subset \r{St}_{\hat\sfP,\hat\sfQ}$ admits the following moduli interpretation: For a classical scheme $S$, the groupoid $\r{St}_{\hat\sfP,\hat\sfQ}(S)$ consists of quadruples $(\sV,\sV_1,\sV_{n-1},u)$ in which
\begin{itemize}
  \item $\sV$ is a vector bundle on $S$ of rank $n$;

  \item $\sV_i\subset \sV$ is a subbundle of rank $i$ (for $i=1,n-1$);

  \item $u\colon\sV\to\sV$ is an automorphism preserving both $\sV_1$ and $\sV_{n-1}$ such that the induced automorphisms on $\sV_1,\sV/\sV_1,\sV_{n-1},\sV/\sV_{n-1}$ are all identities.
\end{itemize}
It follows that $(\r{St}_{\hat\sfP,\hat\sfQ})_{\r{cl}}$ is a union of two closed substacks:
\begin{itemize}
  \item $\r{St}_{\hat\sfP,\hat\sfQ}^{\circ}=(\{1\}/\hat\sfP)\times_{\{1\}/\hat\sfG}(\{1\}/\hat\sfQ)\subset \r{St}_{\hat\sfP,\hat\sfQ}$ consisting of those with $u=1$;

  \item $\r{St}_{\hat\sfP,\hat\sfQ}^{\bullet}=(\sfU_{\hat\sfP}\cap \sfU_{\hat\sfQ})/(\hat\sfP\cap \hat\sfQ)\subset \r{St}_{\hat\sfP,\hat\sfQ}$ consisting of those with $\sV_1\subset \sV_{n-1}$.
\end{itemize}
In addition, the intersection $\r{St}_{\hat\sfP,\hat\sfQ}^{\circ}\cap \r{St}_{\hat\sfP,\hat\sfQ}^{\bullet}$ can be identified with the quotient by $\hat\sfG$ of the incidence divisor $\hat\sfG/(\hat\sfP\cap\hat\sfQ)\subset \hat\sfG/\hat\sfP\times\hat\sfG/\hat\sfQ$. We have the following short exact sequence of coherent sheaves on $\r{St}_{\hat\sfP,\hat\sfQ}$
\begin{equation}\label{eq:zhu6}
0\to \sO_{\r{St}_{\hat\sfP,\hat\sfQ}^{\circ}}(\delta_{\hat\sfP}-\delta_{\hat\sfG,\hat\sfP},\delta_{\hat\sfG,\hat\sfQ}-\delta_{\hat\sfQ})\to \sO_{(\r{St}_{\hat\sfP,\hat\sfQ})_{\r{cl}}}\to \sO_{\r{St}_{\hat\sfP,\hat\sfQ}^{\bullet}}\to 0.
\end{equation}

Notice that the morphism $\fc_p$ from Remark \ref{rem: XG as fixed point stack} restricts to an isomorphism $\sfU_{\hat\sfP,\dZ_\ell}\simeq \sfU_{\hat\sfQ,\dZ_\ell}$ hence induces an isomorphism $(\sfU_{\hat\sfP}/\hat\sfP)_{\dZ_\ell}\simeq (\sfU_{\hat\sfQ}/\hat\sfQ)_{\dZ_\ell}$, again denoted by $\fc_p$. Therefore we can rewrite $\r{St}_{\hat\sfP,\hat\sfQ,\dZ_\ell}$ as
\[
\r{St}_{\hat\sfP,\hat\sfQ,\dZ_\ell}\simeq \hat\sfG/\hat\sfG\times_{\r{id}\times\fc_p, (\hat\sfG/\hat\sfG\times  \hat\sfG/\hat\sfG)}((\sfU_{\hat\sfP}/\hat\sfP)_{\dZ_\ell}\times (\sfU_{\hat\sfP}/\hat\sfP)_{\dZ_\ell}).
\]
If we put
\[
\widetilde\cX_{\sfG}\coloneqq\sfU_{\hat\sfP}/\hat\sfP\times_{\Delta, (\sfU_{\hat\sfP}/\hat\sfP\times \sfU_{\hat\sfP}/\hat\sfP)} \r{St}_{\hat\sfP,\hat\sfQ,\dZ_\ell},
\]
then we have the following commutative diagram
\[
\xymatrix{
\cX_{\sfG}\ar[d] & \ar_{\pi}[l] \widetilde\cX_{\sfG} \ar^{\delta}[r] \ar[d] & \r{St}_{\hat\sfP,\hat\sfQ,\dZ_\ell} \ar[d] \\
 (\hat\sfG/\hat\sfG)_{\dZ_\ell}& \ar[l] (\sfU_{\hat\sfP}/\hat\sfP)_{\dZ_\ell}\ar^-{\Delta}[r] & (\sfU_{\hat\sfP}/\hat\sfP)_{\dZ_\ell}\times (\sfU_{\hat\sfP}/\hat\sfP)_{\dZ_\ell},
}
\]
with Cartesian squares, where the top row is the base change of the bottom arrow along the morphism $\r{id}\times\fc_p\colon\hat\sfG/\hat\sfG\to\hat\sfG/\hat\sfG\times  \hat\sfG/\hat\sfG$.

Put
\[
\widetilde\cX_G^{\circ}\coloneqq\widetilde\cX_G\times_{\r{St}_{\hat\sfP,\hat\sfQ,\dZ_\ell}}\r{St}_{\hat\sfP,\hat\sfQ,\dZ_\ell}^{\circ},\quad \widetilde\cX_G^{\bullet}\coloneqq\widetilde\cX_G\times_{\r{St}_{\hat\sfP,\hat\sfQ,\dZ_\ell}}\r{St}_{\hat\sfP,\hat\sfQ,\dZ_\ell}^{\bullet},\quad \widetilde\cX_G^{\mix}\coloneqq\widetilde\cX_G\times_{\r{St}_{\hat\sfP,\hat\sfQ,\dZ_\ell}}(\r{St}_{\hat\sfP,\hat\sfQ,\dZ_\ell})_{\r{cl}},
\]
which are highly derived stacks. Nevertheless, the geometry of their underlying classical stacks is easy to understand. For example,
\begin{itemize}
  \item $(\widetilde\cX_G^{\circ})_{\r{cl}}=\dP(\sS_\sfG\res_{\cX_\sfG^\unr})$;

  \item generically over $\cX_{\sfG}^{\unr}$, $(\widetilde\cX_\sfG^\bullet\cap\widetilde\cX_\sfG^\circ)_{\r{cl}}$ is a quadric in $\dP(\sS_\sfG\res_{\cX_\sfG^\unr})$;

  \item the morphism $\pi\colon\widetilde\cX_\sfG^\mix\to\cX_\sfG$ restricts to an isomorphism $\widetilde\cX_{\sfG}^{\mix}\setminus \widetilde\cX_\sfG^\circ \xrightarrow{\simeq} \cX_\sfG^\mix\setminus\cX_\sfG^\unr$.

\end{itemize}

Taking the derived structure into account, we have the following result. In what follows, tensor products of coherent sheaves are always taken over the corresponding structural sheaf.

\begin{lem}\label{lem: coherent sheaf for balloon stratum}
The stack $\widetilde\cX_\sfG^{\circ}$  is the odd relative cotangent bundle of $\dP(\sS_\sfG\res_{\cX_\sfG^\unr})\to \cX_\sfG^\unr$. More precisely, its underlying classical stack is the projective bundle associated with the vector bundle $\sS_\sfG\res_{\cX_\sfG^\unr}$, but the structural sheaf is the symmetric algebra of the relative tangent sheaf of $\dP(\sS_\sfG\res_{\cX_\sfG^\unr})\to \cX_\sfG^\unr$ placed in cohomological degree $-1$.
\end{lem}
\begin{proof}
We sketch the proof. Keeping track of the definitions and various identifications explained above, we see that
\[
\widetilde\cX_\sfG^\unr\simeq
\(\{1\}/\hat\sfP\times_{\{1\}/\hat\sfG}\{1\}/\hat\sfQ\)\times_{\sfU_{\hat\sfP}/\hat\sfP\times \sfU_{\hat\sfP}/\hat\sfP, \Delta} \sfU_{\hat\sfP}/\hat\sfP.
\]
Now the claim follows from the Koszul resolution of the diagonal $\sfU_{\hat\sfP}\xrightarrow{\Delta}\sfU_{\hat\sfP}\times \sfU_{\hat\sfP}$ and the fact that $\hat\sfG\times^{\hat\sfP}(\Lie \sfU_{\hat\sfP})^*$ can be identified with the tangent bundle of $\hat\sfG/\hat\sfP$.
\end{proof}

\begin{remark}
We invite readers to find the formal analogy between the geometry of $\rM^{\loc}$ and $\widetilde\cX^{\mix}_G$.
\end{remark}

Now we can construct coherent sheaves. By definition, we have the universal filtration
\begin{align}\label{eq:zhu7}
0\to\sP\to\pi^*\sS_\sfG\to\sQ\to 0
\end{align}
of vector bundles on $\widetilde\cX_\sfG$ (which is the descent of the obvious filtration from $\cY_\sfG\times_{\hat\sfG_{\dZ_\ell}}\r{Spr}_{\hat\sfP,\dZ_\ell}$). Note that $\sQ\simeq \delta^*\sO(\delta_{\hat\sfG,\hat\sfP}-\delta_{\hat\sfP},0)\simeq \delta^*\sO(0,\delta_{\hat\sfQ}-\delta_{\hat\sfG,\hat\sfQ})$. The short exact sequence \eqref{eq:zhu6} induces the following cofiber sequence
\begin{align}\label{eq:zhu5}
\sO_{\widetilde\cX_\sfG^\circ}\otimes\sQ^{-1-n}\otimes \pi^*\sD_\sfG \to \sO_{\widetilde\cX_\sfG^\mix}\otimes\sQ^{1-n}\otimes  \pi^*\sD_\sfG
\to \sO_{\widetilde\cX_\sfG^\bullet}\otimes\sQ^{1-n}\otimes \pi^*\sD_\sfG
\end{align}
of coherent sheaves on $\widetilde\cX_\sfG$.  

Recall that the unipotent categorical local Langlands correspondence $\dL_\sfG$ has been established in \cite{Zhu25} for $\ol\dQ_\ell$-coefficients and, under some mild restriction of $\ell$, for $\ol\dF_\ell$-coefficients. To state the following conjecture, we base change all the $\dZ_\ell$-sheaves and all the stacks over $\dZ_\ell$ appeared in the previous discussions to the coefficient field $\dE$ that is either $\ol\dQ_\ell$ or $\ol\dF_\ell$, although we expect that the $\dZ_\ell$-version will continue to hold. The following conjecture should be within reach.

\begin{conjecture}\label{conj: CLLC}
We regard $\cind_{\sfK}^{\sfG(\dQ_p)}\dE$, $\cind_{\sfK^\bullet}^{\sfG(\dQ_p)}\dE$, and $\cind_{\sfK'}^{\sfG'(\dQ_p)}\dE$ as sheaves on $\r{Isoc}_{\sfG}$, supported on $[\Spec\dF_p/\sfG(\dQ_p)]$, $[\Spec\dF_p/\sfG(\dQ_p)]$, and $[\Spec\dF_p/\sfG'(\dQ_p)]$, respectively.
\begin{enumerate}
  \item Under the unipotent categorical local Langlands correspondence $\dL_{\sfG}$ (with appropriate choice of Whittaker datum), we have
      \begin{align*}
      \sA_{\sfG,\sfK}&\coloneqq \dL_{\sfG}(\cind_{\sfK}^{\sfG(\dQ_p)}\dE)\simeq \sO_{\cX_G^{\unr}}, \\
      \sA_{\sfG',\sfK'}&\coloneqq \dL_{\sfG}(\cind_{\sfK'}^{\sfG'(\dQ_p)}\dE)\simeq \rR\pi_*\(\sO_{\widetilde\cX_\sfG^\mix}\otimes\sQ^{1-n}\)\otimes \sD_\sfG;
      \end{align*}
      and
      \begin{align*}
      \sA_{\sfG,\sfK^\bullet}\coloneqq \dL_{\sfG}(\cind_{\sfK^\bullet}^{\sfG(\dQ_p)}\dE)\simeq
      \begin{dcases}
      \sA_{\sfG,\sfK}\otimes\sD_\sfG(\tfrac{n(1-n)}{2}), &\text{$n$ even}; \\
      \sA_{\sfG',\sfK'}\otimes\sD_\sfG(\tfrac{n(1-n)}{2}), &\text{$n$ odd}.
      \end{dcases}
      \end{align*}
      In addition, $\sA_{\sfG',\sfK'}$ is a maximal Cohen--Macaulay coherent sheaf (concentrated in degree zero). When $n$ is odd, it is supported on $\cX_{\sfG}^{\mix}$, and its restriction to $\cX_{\sfG}^{\unr}\setminus \cX_{\sfG}^{\r{cong}}$ is isomorphic to $\sD_{\sfG}(\tfrac{n(1-n)}{2})\res_{\cX_{\sfG}^{\unr}\setminus \cX_{\sfG}^{\r{cong}}}$. When $n$ is even, it is supported on $\cX_{\sfG}^{\mnm}$, and is of generic rank one.


  \item Under the unipotent categorical local Langlands correspondence $\dL_{\sfG}$, the local counterpart of \eqref{eq:zhu1} (with coefficients in $\dE$) after applying $\r{Nt}_*$ corresponds to the pushforward of \eqref{eq:zhu7} after tensoring with $\sO_{\widetilde\cX_\sfG^\mix}\otimes\sQ^{1-n}\otimes\pi^*\sD_\sfG$, that is, the cofiber sequence
      \begin{align*}
      \rR\pi_*\(\sO_{\widetilde\cX_\sfG^\mix}\otimes\sQ^{1-n}\otimes\sP\)\otimes \sD_\sfG \to \sA_{\sfG',\sfK'}\otimes\sS_\sfG
      \to \rR\pi_*\(\sO_{\widetilde\cX_\sfG^\mix}\otimes\sQ^{2-n}\)\otimes \sD_\sfG.
      \end{align*}

  \item Under the unipotent categorical local Langlands correspondence $\dL_{\sfG}$, the local counterpart of \eqref{eq:zhu2} (with coefficients in $\dE$) after applying $\r{Nt}_*$ corresponds to the pushforward of \eqref{eq:zhu5} after tensoring with $\sQ$, that is, the cofiber sequence
      \begin{align*}
      \rR\pi_*\(\sO_{\widetilde\cX_\sfG^\circ}\otimes\sQ^{-n}\)\otimes \sD_\sfG \to \rR\pi_*\(\sO_{\widetilde\cX_\sfG^\mix}\otimes\sQ^{2-n}\)\otimes \sD_\sfG
      \to \rR\pi_*\(\sO_{\widetilde\cX_\sfG^\bullet}\otimes\sQ^{2-n}\)\otimes \sD_\sfG.
      \end{align*}

\end{enumerate}
\end{conjecture}


Assuming this conjecture, let us outline alternative proofs of some results in previous sections, following some ideas of \cite{Yan25}. We need the proposition below as an additional input.

\begin{proposition}\label{prop: Coh of Shimura variety}
There is an (ind-)coherent complex $\Igs^{\r{spec}}$ on $\cX_{\sfG}$, equipped with the prime-to-$p$ Hecke algebra action, such that for every $\ell$-adic constructible sheaf $F$ (with coefficients in $\dE$) on $\rM$ obtained as the ($!$-)pullback of a sheaf $F^{\loc}$ on $\rM^{\loc}\subset \r{Sht}^{\loc}_{\{0,n-1\}}$ that appears in the local counterpart of \eqref{eq:zhu1}, \eqref{eq:zhu2}, or \eqref{eq:zhu3}, there is a canonical Hecke-equivariant isomorphism
\[
\rR\Gamma(\overline\rM, \mathbb DF)\simeq \rR\!\Hom_{\cX_G}\(\dL_{\sfG}(\r{Nt}_*F^\loc), \Igs^{\r{spec}}\),
\]
where $\dD$ denotes the Verdier duality operator.
\end{proposition}

\begin{proof}
This is a consequence of \cite{Zhu25}*{Theorem 6.16} and \cite{YZ}*{Theorem 4.20}.
\end{proof}


\begin{example}\label{example: Coh of Shimura variety}
For examples, we have the following more precise descriptions.
\begin{enumerate}

  \item When $F=\rR\Psi(\dE[n-1])$ in the above proposition, we can identify the left-hand side with the cohomology of the geometric generic fiber of $\bM$ (up to shifts and Tate twists), and therefore is equipped with an action of the (parahoric) Hecke algebra for $\sfK'$. On the other hand, since $\dL_{\sfG}(\r{Nt}_*F^\loc)\simeq\sA_{\sfG',\sfK'}\otimes\sS_{\sfG}$, the right-hand side is also equipped with an action of the Hecke algebra for $\sfK'$, through its action on $\sA_{\sfG',\sfK'}$. Then the isomorphism in the proposition is also compatible with the action of this Hecke algebra on both sides by \cite{YZ}*{Theorem 4.20}. In addition, under this isomorphism, the Galois action on the left-hand side corresponds to the universal action $\rho_{\univ}$ on $\sS_{\sfG}$ on the right-hand side.

  \item When $F=\dE_{\rM^{\circ}}[n-1]$ in the above proposition, we can identify the left-hand side with the cohomology of $\ol\rM^{\circ}$ (up to shifts and Tate twist), which can then be identified with $\bigoplus_{j=0}^{n-1} \dE[\rS^\circ][1-n+2j](j)$. On the other hand, using Lemma \ref{lem: coherent sheaf for balloon stratum} and Conjecture \ref{conj: CLLC}, the right-hand side is nothing but
      \[
      \rR\!\Hom_{\cX_G}\(\bigoplus_{j=0}^{n-1}\sO_{\cX_{\sfG}^{\unr}}[1-n+2j](j), \Igs^{\r{spec}}\).
      \]
      It follows that
      \[
      \dE[\rS^\circ]\simeq \rR\!\Hom_{\cX_G}\(\sO_{\cX_{\sfG}^{\unr}}, \Igs^{\r{spec}}\).
      \]
      When $n$ is even and let $\tN$ be the Hecke operator as in \eqref{eq:deformation_1} and Example \ref{ex: Satake for min}. Then
      \[
      \dE[\rS^\circ]/\tN\dE[\rS^\circ]\simeq\rH^1\rR\!\Hom_{\cX_G}\(\sO_{\cX_{\sfG}^{\r{cong}}}, \Igs^{\r{spec}}\).
      \]
\end{enumerate}
\end{example}

In the sequel, for a (complex of) $\fZ_{\sfG}^{\unr}$-module $L$, we let $L_\xi$ denote its localization at $\xi$. We also identify $\fZ_{\sfG}^{\unr}$ with the spherical Hecke algebra of $\sfG$ via the Satake isomorphism (after choosing $\sqrt{p}$).

\begin{remark}
Let $\fn$ be a maximal ideal of the prime-to-$p$ Hecke algebra. Then we have the localized (ind-)coherent complex $\Igs^{\r{spec}}_{\fn}$ at $\fn$, and an isomorphism
\[
\rR\!\Hom_{\cX_G}\(\dL_{\sfG}(\r{Nt}_*F^\loc), \Igs^{\r{spec}}\)_\fn\simeq
\rR\!\Hom_{\cX_G}\(\dL_{\sfG}(\r{Nt}_*F^\loc), \Igs^{\r{spec}}_{\fn}\).
\]
Similarly, given a point $\xi\in \Spec \fZ_{\sfG}$, we also have the localized $\Igs^{\r{spec}}_{\xi}$ at $\xi$, and an isomorphism
\[
\rR\!\Hom_{\cX_G}\(\dL_{\sfG}(\r{Nt}_*F^\loc), \Igs^{\r{spec}}\)_{\xi}\simeq
\rR\!\Hom_{\cX_G}\(\dL_{\sfG}(\r{Nt}_*F^\loc), \Igs^{\r{spec}}_{\xi}\).
\]
\end{remark}

First, suppose that $n=2r+1$ is odd. Let $U\subset \cX_{\sfG}^{\unr}$ be the open substack classifying unramified $L$-parameter $\rho$ whose associated unitary Satake parameter does not contain $\{-p,-p^{-1}\}$ and contains $1$ exactly once. Note that  $U$ is disjoint from $\cX_{\sfG}^{\mnm}$ (by Example \ref{ex: Satake for min}) and therefore is also open in $\cX_{\sfG}^{\mix}$. By Conjecture \ref{conj: CLLC}, we have $\sA_{\sfG',\sfK'}\res_U\simeq \sD_{\sfG}(-nr)\res_U$. On the other hand, the sheaf $\widetilde\sS\coloneqq\sS_{\sfG}(r)\res_U$ is equipped with the representation $\widetilde\rho=\rho_{\univ}(r)$. Then the action of $\widetilde\rho(\phi)$ induces a decomposition
\[
\widetilde\sS\simeq \widetilde{\sS}^1\oplus \widetilde{\sS}^{\neq 1}
\]
in the way that $\widetilde\rho(\phi)$ acts on $\widetilde\sS^1\simeq\sD_{\sfG}(-nr)\res_U$ trivially and that $\widetilde\rho(\phi)-1$ acts on $\widetilde\sS^{\neq 1}$ by an automorphism. Taking Remark \ref{rem: determinant of tautological representation} into account, it follows that we have the following cofiber sequence
\[
\sA_{\sfG',\sfK'}\otimes \widetilde\sS\xrightarrow{\widetilde\rho(\phi)-1} \sA_{\sfG',\sfK'}\otimes \widetilde\sS\to \sO_{U}\oplus \sO_U[1].
\]
Combining these observations with Proposition \ref{prop: Coh of Shimura variety} and Example \ref{example: Coh of Shimura variety}, we obtain the following.

\begin{corollary}
Let $\xi\in \Spec \fZ_{\sfG}^{\unr}$ be a point such that the corresponding unitary Satake parameter does not contain $\{-p,-p^{-1}\}$ and contains $1$ exactly once. Then the degree zero cohomology of the cofiber of
\[
\r{R}\Gamma(\rM'\otimes_{\dQ_{p^2}}\overline\dQ_{p}, \dE[2r](r))_{\xi}\xrightarrow{1-\phi} \r{R}\Gamma(\rM'\otimes_{\dQ_{p^2}}\overline\dQ_{p}, \dE[2r](r))_{\xi}
\]
is canonically isomorphic to $\dE[\rS^{\circ}]_\xi$.
\end{corollary}

Now suppose that we are in the situation of Proposition \ref{pr:tate}. The restriction of the residual representation the Galois representation $V_{1,\lambda}^{\Pi}$ to the decomposition group at $\fp$ is unramified, giving rise to a point $\xi$ as in the above corollary. One checks that
\[
\rH^{n-1}(\rM'\otimes_{\dQ_{p^2}}\overline\dQ_{p},\ol\dF_\ell(r))_{\fn}=\rH^{n-1}(\rM'\otimes_{\dQ_{p^2}}\overline\dQ_{p}, \ol\dF_\ell(r))_{(\fn,\xi)}\subset \rR\Gamma(\rM'\otimes_{\dQ_{p^2}}\overline\dQ_{p},\ol\dF_\ell[2r](r))_{\xi}
\]
is a direct summand. Then taking $\fn$-component of the above corollary recovers part of Proposition \ref{pr:tate} (for coefficients in $\ol\dF_\ell$).\footnote{Note that this argument in fact does not require the assumption that $V^{\Pi}_{1,\lambda}$ is absolutely irreducible.}

Next, suppose that $n$ is even and assume $\ell\neq p^2-1$. Let $U\subset \cX_{\sfG}^{\mnm}$ be the open substack classifying those $\rho$ such that the unitary Satake parameter associated with $\rho^{\r{ss}}$ contains $\{p,p^{-1}\}$ exactly once; put $U^{\r{cong}}=U\cap \cX_{\sfG}^{\r{cong}}$. Write $\widetilde\sS\coloneqq\sS^\vee_{\sfG}(1-r)\res_U$ for simplicity, which is equipped with the representation $\widetilde\rho=\rho_{\univ}^\vee(r-1)$. Then we similarly have a decomposition (depending on the choice of a lifting of the arithmetic Frobenius $\phi$ of $\dQ_{p^2}$)
\begin{equation}\label{eq: dec of universal rep over Xmin}
\widetilde\sS=\widetilde\sS^1\oplus \widetilde\sS^{p^{-2}}\oplus \widetilde\sS^{\not\in\{1,p^{-2}\}}
\end{equation}
according to the action of $\widetilde\rho(\phi)$. In addition, the operator $\sfx\coloneqq\widetilde\rho(\tau)-1$ acts trivially on $\widetilde\sS^{1}\oplus \widetilde\sS^{\not\in\{1,p^{-2}\}}$ trivially and induces an injection
\[
\widetilde\sfx\colon\widetilde\sS^{p^{-2}}\to \widetilde\sS^{1}.
\]
Write $\sL\coloneqq\widetilde\sS^{1}$ for short. Then $\widetilde{\sS}^{p^{-2}}=\sL^\vee(1)$. As the ideal sheaf of $\cX_{\sfG}^{\r{cong}}\subset \cX_{\sfG}^{\mnm}$ is the image of the map $(\sL^\vee)^{\otimes 2}(1)\to\sO_{U}$, the cokernel of $\widetilde\sfx$ is $\sL\res_{U^{\r{cong}}}$.

Recall that $U$ is smooth, which implies that $\sA_{\sfG',\sfK'}\res_U$ is a line bundle (since it is maximal Cohen--Macaulay of generic rank one by Conjecture \ref{conj: CLLC}(1)). In addition, it is self-dual with respect to the Grothendieck--Serre duality. It follows that $\sA_{\sfG',\sfK'}\res_U\simeq \sL^\vee\otimes\sD_{\sfG}(r(n-1))\res_U$.

Let $\xi\in \fZ_{\sfG}$ be a point whose unitary Satake parameter contains $\{p,p^{-1}\}$ exactly once.
The decomposition \eqref{eq: dec of universal rep over Xmin} induces the corresponding decomposition (still depending on the choice of a lifting of $\phi$) of the cohomology
\begin{align*}
\rR\Gamma(\rM'\otimes_{\dQ_{p^2}}\overline\dQ_{p},\ol\dF_\ell[n-1](r))_{\xi}&=
\r{R}\Gamma(\rM'\otimes_{\dQ_{p^2}}\overline\dQ_{p}, \ol\dF_\ell[n-1](r))^1_{\xi}\\
&\oplus \r{R}\Gamma(\rM'\otimes_{\dQ_{p^2}}\overline\dQ_{p}, \ol\dF_\ell[n-1](r))^{p^{-2}}_{\xi} \\
&\oplus \r{R}\Gamma(\rM'\otimes_{\dQ_{p^2}}\overline\dQ_{p}, \ol\dF_\ell[n-1](r))_{\xi}^{\not\in\{1,p^{-2}\}}.
\end{align*}
Now assume that $\r{R}\Gamma(\rM'\otimes_{\dQ_{p^2}}\overline\dQ_{p},\ol\dF_\ell[n-1](r))_{\xi}$ is concentrated in degree zero. Then it is clear that
\[
\rH^1_{\sing}(\dQ_{p^2},\rH^{n-1}(\rM'\otimes_{\dQ_{p^2}}\overline\dQ_{p}, \ol\dF_\ell(r))_{\xi})= \frac{\rH^{n-1}(\rM'\otimes_{\dQ_{p^2}}\overline\dQ_{p}, \ol\dF_\ell(r))^1_{\xi}}{\sfx.\rH^{n-1}(\rM'\otimes_{\dQ_{p^2}}\overline\dQ_{p}, \ol\dF_\ell(r))^{p^{-2}}_{\xi}},
\]
which can be computed as the degree $1$-term of the complex
\[
\rR\!\Hom_{\cX_G}\(\sA_{\sfG',\sfK'}\otimes (\widetilde\sS^{p^{-2}})^\vee/\,\widetilde\sfx^\vee(\widetilde\sS^{1})^\vee, \Igs^{\r{spec}}\)
\simeq
\rR\!\Hom\(\sO_{U^{\r{cong}}}\otimes \sD_{\sfG}(r(1-n)),  \Igs^{\r{spec}}\).
\]
If in addition the unitary Satake parameter of $\xi$ does not contain $-1$, then $\sD_{\sfG}(r(1-n))$ is isomorphic to $\sO_U$ around $\xi$. Therefore, the desired cohomology can be identified with $\ol\dF_\ell[\rS^{\circ}]_{\xi}/\tN\ol\dF_\ell[\rS^{\circ}]_{\xi}$, by Example \ref{example: Coh of Shimura variety}. This reproves part of Proposition \ref{pr:singular} (for coefficients in $\ol\dF_\ell$), which avoids using R=T type results hence removes some extra assumptions.

\end{document}